\documentclass{article}
\usepackage{amssymb}
\usepackage{amsfonts}
\usepackage{amsmath}

\newtheorem{theorem}{Theorem}

\newtheorem{corollary}[theorem]{Corollary}

\newtheorem{definition}[theorem]{Definition}
\newtheorem{example}[theorem]{Example}

\newtheorem{lemma}[theorem]{Lemma}
\newtheorem{notation}[theorem]{Notation}

\newtheorem{proposition}[theorem]{Proposition}
\newtheorem{remark}[theorem]{Remark}

\newenvironment{proof}[1][Proof]{\noindent\textbf{#1.} }{\ \rule{0.5em}{0.5em}}
\begin{document}

\title{Essential Circles and Gromov-Hausdorff Convergence of Covers}
\author{Conrad Plaut \\
Department of Mathematics\\
University of Tennessee\\
Knoxville, TN 37996\\
cplaut@math.utk.edu \and Jay Wilkins \\
Department of Mathematics\\
University of Connecticut\\
196 Auditorium Road, Unit 3009\\
Storrs, CT 06269-3009\\
leonard.wilkins@uconn.edu}
\maketitle

\begin{abstract}
We give various applications of essential circles (introduced in \cite{PW})
in a compact geodesic space $X$. Essential circles completely determine the
homotopy critical spectrum of $X$, which we show is precisely $\frac{2}{3}$
the covering spectrum of Sormani-Wei. We use finite collections of essential
circles to define \textquotedblleft circle covers\textquotedblright , which
extend and contain as special cases the $\delta $-covers of Sormani and Wei
(equivalently the $\varepsilon $-covers of \cite{PW}); the constructions are metric
adaptations of those utilized by Berestovskii-Plaut in the construction of
entourage covers of uniform spaces. We show that, unlike $\delta$- and 
$\varepsilon $-covers, circle covers are in a sense closed with respect to
Gromov-Hausdorff convergence, and we prove a finiteness theorem concerning
their deck groups that does not hold for covering maps in general. This
allows us to completely understand the structure of Gromov-Hausdorff limits
of $\delta$-covers.  Also, we use essential circles to strengthen a theorem
of E. Cartan by finding a new (even for compact Riemannian manifolds) finite
set of generators of the fundamental group of a semilocally simply connected
compact geodesic space. We conjecture that there is always a generating set
of this sort having minimal cardinality among all generating sets.

\end{abstract}

\section{Introduction}

Christina Sormani and Guofang Wei in \cite{SW1}, \cite{SW2}, \cite{SW3}, 
\cite{SW4} and Valera Berestovskii and Conrad Plaut in \cite{BPTG}, \cite%
{BPUU} studied covering space constructions that encode geometric
information and stratify the topology of the underlying space. Sormani and
Wei utilized a classical construction of Spanier (\cite{S}) that provides a
covering map $\pi ^{\delta }:\widetilde{X}^{\delta }\rightarrow X$
corresponding to the open cover of a geodesic space $X$ by open $\delta $%
-balls, which they called the $\delta $-cover of $X$. Berestovskii-Plaut
developed a new construction for covers of uniform spaces -- hence any
metric space -- that utilizes discrete chains and homotopies rather than
traditional (i.e. continuous) paths and homotopies. Building on
Berestovskii-Plaut, the special case of metric spaces was developed further
by Plaut and Wilkins in \cite{PW} (see also \cite{PWetal} and \cite{W}). For
any connected metric space $X$ and $\varepsilon >0$, the Berestovskii-Plaut construction
yields a covering map $\phi _{\varepsilon }:X_{\varepsilon }\rightarrow X$,
with deck group $\pi _{\varepsilon }(X)$ that is a kind of coarse
fundamental group at a given scale. In the geodesic case, which is the case
of interest in this paper, we show that when $\delta =\frac{3\varepsilon }{2}
$, the covering maps $\pi ^{3\varepsilon /2}$ and $\phi _{\varepsilon }$ are
naturally isometrically equivalent (Corollary \ref{sw}). Both covering
spaces have been used to obtain notable geometric and topological results
(cf. \cite{PWetal}, \cite{PW}, \cite{SW1}, \cite{SW2}, \cite{SW3}, \cite{SW4}%
, \cite{W}).

As Sormani-Wei observed in \cite{SW1}, one characteristic of the $\delta $%
-covers of a geodesic space -- hence, the $\varepsilon$-covers, also -- is that 
they are not \textquotedblleft closed\textquotedblright\ with respect to 
Gromov-Hausdorff convergence. That is, if a sequence of compact geodesic spaces 
$X_{i}$ converges to a compact space $X$ in the Gromov-Hausdorff sense, the covers 
$(X_{i})_{\varepsilon }$ may not converge to $X_{\varepsilon }$ in the pointed 
Gromov-Hausdorff sense. Even if the covers $(X_{i})_{\varepsilon }$ converge, 
their limit may not be any $\varepsilon $-cover of $X$. Sormani and Wei did show 
that there is always a subsequence $(X_{i_{k}})_{\varepsilon }$ that pointed converges
to a space $Y$, that $Y$ is covered by $X_{\varepsilon }$, and that $Y$
covers both $X$ and $X_{\varepsilon ^{\prime }}$ for all $\varepsilon
^{\prime }>\varepsilon $ (Theorem 3.6, \cite{SW1} and Proposition 7.3, \cite%
{SW2}). It follows that the $\varepsilon $-covers of a convergent sequence
of compact geodesic spaces are \textit{precompact}. In this paper we
introduce the notion of \textquotedblleft circle cover\textquotedblright ,
which extends the notion of $\varepsilon $-cover and is closed with respect
to Gromov-Hausdorff convergence. As a consequence we can fully understand
the limit space $Y$ described by Sormani-Wei. 

Discrete homotopy methods are quite amenable to questions that involve
Gromov-Hausdorff convergence. For example, Gromov-Hausdorff convergence of
compact metric spaces can be characterized by the existence of
\textquotedblleft almost isometries\textquotedblright\ that generally are
not continuous, and therefore classical methods using continuous paths are
not always easy to apply. However, for any function $f:X\rightarrow Y$
between metric spaces, there is a naturally induced function $%
f_{\#}:X_{\varepsilon }\rightarrow Y_{\delta }$ provided that $f$ only
satisfies a kind of coarse continuity: for any $x,y\in X$, if $%
d(x,y)<\varepsilon $ then $d(f(x),f(y))<\delta $ (see also Theorem 2 in \cite%
{BPUU}). In fact, we show that when $f$ is a generalized almost isometry
called a $\sigma $-isometry, $f_{\#}$ is actually a quasi-isometry (see
Remark \ref{quasi} and Theorem \ref{almstdone}) with distortion constants
explicitly controlled by $\sigma $. These tools play a significant role in
the proof of our main theorem (Theorem \ref{closed}), as these induced
quasi-isometries allow us to gain explicit control over the Gromov-Hausdorff
distance between balls in the possibly non-compact spaces $X_{\varepsilon }$
and $Y_{\delta }$. 

For simplicity, in what follows we will use the notation of \cite{PW}
concerning $\varepsilon $-covers, although the results that apply to
geodesic spaces can be equivalently stated for the $\delta $-covers of
Sormani-Wei, due to Corollaries \ref{sw} and \ref{lker}. In \cite{PW} we 
defined the notion of an \textit{essential} $\varepsilon$\textit{-circle} in a 
geodesic space (see Sections 2 and 3 for more details). Among their other 
applications, essential circles completely determine the Sormani-Wei 
\textit{Covering Spectrum }(equivalently the \textit{Homotopy Critical Spectrum} 
of \cite{PW}). Now suppose that $\varepsilon >0$ and $\mathcal{T}$ is any 
finite collection of essential $\delta $-circles such that $\delta \geq \varepsilon$. 
We define a natural normal subgroup $K_{\varepsilon }(\mathcal{T})$ of 
$\pi _{\varepsilon }(X)$ (which is the trivial group when $\mathcal{T=\varnothing }$) 
that acts freely and properly discontinuously on $X_{\varepsilon }$ with quotient $%
X_{\varepsilon }^{\mathcal{T}}$. Then there is a natural induced mapping $%
\phi _{\mathcal{\varepsilon }}^{\mathcal{T}}:X_{\mathcal{\ \varepsilon }}^{%
\mathcal{T}}\rightarrow X$ which is also a covering map with deck group
naturally isomorphic to $\pi _{\mathcal{\varepsilon }}^{\mathcal{T}}(X)=\pi
_{\varepsilon }(X)/K_{\varepsilon }(\mathcal{T)}$. We call $\phi
_{\varepsilon }^{\mathcal{T}}$ the $(\mathcal{T},\varepsilon )$-cover of $X$,
and in general we will refer to these covers as \textit{circle covers} (Definition \ref{circle cover}). 
Note that when $\mathcal{T=\varnothing }$, $\phi _{\varepsilon }^{\mathcal{T}%
}=\phi _{\varepsilon }$, so circle covers extend the notion of $\varepsilon $
-covers. In this paper, the arrow \textquotedblleft $\rightarrow $%
\textquotedblright\ indicates Gromov-Hausdorff convergence (respectively,
pointed Gromov-Hausdorff convergence) of the compact spaces $X_{i}$
(respectively, for the possibly non-compact covering spaces). 

\begin{theorem}
\label{closed}Suppose that $X_{i}\rightarrow X$, where each $X_{i}$ is
compact geodesic, and for each $i$ there are an $\varepsilon _{i}>0$ and a
finite collection $\mathcal{T}_{i}$ of essential $\tau $-circles in $X_{i}$
such that $\tau \geq \varepsilon _{i}$. If $\{\varepsilon _{i}\}$ has a
positive lower bound then for any positive $\varepsilon \leq \lim \inf
\varepsilon _{i}$ there exist a subsequence $\{X_{i_{k}}\}$ and a finite
collection $\mathcal{T}$ of essential $\tau $-circles in $X$ with $\tau \geq
\varepsilon $, such that $(X_{i_{k}})_{\varepsilon _{i_{k}}}^{\mathcal{T}%
_{i_{k}}}\rightarrow X_{\varepsilon }^{\mathcal{T}}$ and $\pi _{\varepsilon
_{i_{k}}}^{\mathcal{T}_{i_{k}}}(X_{i_{k}})$ is isomorphic to $\pi
_{\varepsilon }^{\mathcal{T}}(X)$ for all large $k$.
\end{theorem}

Note that without a positive lower bound on the size of the essential
circles the situation is more complicated; in general, there may be no
convergent subsequence of covers at all, and even if there is, it may not be
a cover of the limiting space $X$ (see Example \ref{zero limit ex}).  Nevertheless,
Theorem \ref{closed} is an explicit enhancement of some notable prior results.
For instance, from Theorem 3.4 and Corollary 3.5 in \cite{SW1}, one can conclude 
that for any $0 < \varepsilon < \varepsilon'$, $\pi_{\varepsilon'}(X)$ is isomorphic to 
a quotient of $\pi_{\varepsilon}(X_i)$ for sufficiently large $i$.  Theorem \ref{closed},
on the other hand, not only gives a nice geometric picture of the explicit quotient, but 
it also handles the case where $\varepsilon = \varepsilon'$

In Section 2 we recall some of the basics of discrete homotopy theory from 
\cite{BPUU} and \cite{PW}, and present examples illustrating ideas outlined
in the introduction. In Section 3 we establish some technical results
regarding quotients of metric spaces, particularly in the case of subgroups
of $\pi _{\varepsilon }(X)$ acting on $X_{\varepsilon }$. We establish in
Theorem \ref{mcov} that \textit{every} cover of a compact geodesic space $X$
is naturally a quotient of $X_{\varepsilon }$ by a subgroup of $\pi
_{\varepsilon }(X)$, for some sufficiently small $\varepsilon >0$. In the
process, we give an explicit description of the subgroup as well as
necessary and sufficient conditions for the covering map to be regular; this
strengthens Lemma 2.4 in \cite{SW2}. As a byproduct of the results in
Section 3, we use essential circles to create a new (even for Riemannian
manifolds) set of generators of the fundamental group of a compact geodesic
space (Theorem \ref{newbase}), and we conjecture that one can always find
such a generating set of minimal cardinality.

Section 4 begins by introducing the notion of $\sigma $-isometry and showing
that pointed Gromov-Hausdorff convergence is implied by the existence of
these maps. The induced maps $f_{\#}$ mentioned above play a significant
role in this section, and Proposition \ref{serious} provides the technical
machinery necessary to control the distortion of these maps in the case of
the $\varepsilon $-covers of a convergent sequence of compact spaces. The
last step is the adaptation and translation of the preceding tools into an
equivariant notion of pointed Gromov-Hausdorff convergence (Theorem \ref%
{almstdone} and Proposition \ref{equiclose}); this facilitates the still
rather technical proof of Theorem \ref{closed}. We conclude the paper with
some brief remarks and questions that naturally arise from these results. 

\section{Background and Examples}

We begin with some background on discrete homotopy theory; proofs and 
further details regarding the results in this section may be found in \cite{BPUU}
for the more general uniform case and in \cite{PW} for the special case of metric
and geodesic spaces.  In a metric space $X$, and for $\varepsilon >0$, an 
$\varepsilon $-chain is a finite sequence $\{x_{0},...,x_{n}\}$ such that for 
all $i$, $d(x_{i},x_{i+1})<\varepsilon $.  An $\varepsilon $-homotopy consists 
of a finite sequence $\left\langle \gamma _{0},...,\gamma _{n}\right\rangle$ of 
$\varepsilon $-chains, where each $\gamma _{i}$ differs from its predecessor by 
a \textquotedblleft basic move\textquotedblright : adding or removing a \textit{single} 
point, always leaving the endpoints fixed. The $\varepsilon $-homotopy equivalence class
of an $\varepsilon $-chain $\alpha $ will be denoted by $[\alpha
]_{\varepsilon }$. Fixing a basepoint $\ast $, $X_{\varepsilon }$ is defined
to be the set of all $\varepsilon $-homotopy equivalence classes of $%
\varepsilon $-chains starting at $\ast $, and $\phi _{\varepsilon
}:X_{\varepsilon }\rightarrow X$ is the endpoint map. In a connected space,
choice of basepoint is immaterial and $\phi _{\varepsilon }$ is surjective,
so we will not include the base point in our notation and will assume that
all maps are base-point preserving. In particular, we always take the base
point in $X_{\varepsilon }$ to be the equivalence class $[\ast
]_{\varepsilon }$ containing the trivial chain $\{\ast \}$.

The group $\pi _{\varepsilon }(X)$ is the subset of $X_{\varepsilon }$
consisting of equivalence classes of $\varepsilon $-loops starting and
ending at $\ast $ with operation induced by concatenation, i.e., $[\alpha
]_{\varepsilon }\ast \lbrack \beta ]_{\varepsilon }=[\alpha \ast \beta
]_{\varepsilon }$. We denote the reversal of a chain $\alpha $ by $\overline{%
\alpha }$. As expected, for $[\alpha ]_{\varepsilon }\in \pi _{\varepsilon
}(X)$, $\left( [\alpha ]_{\varepsilon }\right) ^{-1}=[\overline{\alpha }%
]_{\varepsilon }$, and the identity is $[\ast ]_{\varepsilon }$.

For any $\varepsilon $-chain $\alpha =\{x_{0},...,x_{n}\}$, we set $\nu
(\alpha ):=n$ and define its \textit{length} by 
\begin{equation}
L(\alpha ):=\sum\limits_{i=1}^{n}d(x_{i},x_{i-1}).  \notag
\end{equation}%
\noindent Defining $\left\vert [\alpha ]_{\varepsilon }\right\vert :=\inf
\{L(\gamma ):\gamma \in \lbrack \alpha ]_{\varepsilon }\}$ leads to a metric
on $X_{\varepsilon }$ given by 
\begin{equation}
d([\alpha ]_{\varepsilon },[\beta ]_{\varepsilon }):=\left\vert \left[ 
\overline{\alpha }\ast \beta \right] _{\varepsilon }\right\vert =\inf
\{L(\kappa ):\alpha \ast \kappa \ast \overline{\beta }\text{ is }\varepsilon 
\text{-null homotopic}\}  \label{mdeff}
\end{equation}%
This metric has a number of nice properties that we will need. For example, $%
\pi _{\varepsilon }(X)$ acts on $X_{\varepsilon }$ as isometries via the map
induced by preconcatenation by an $\varepsilon $-loop. Additionally, the
endpoint map $\phi _{\varepsilon }:X_{\varepsilon }\rightarrow X$ (which is
well-defined since $\varepsilon $-homotopies preserve endpoints), is a
uniform local isometry and, provided $X$ is connected, a regular covering
map with deck group naturally identified with $\pi _{\varepsilon }(X)$. When 
$X$ happens to be a geodesic space (which will soon be our underlying
assumption) then so is $X_{\varepsilon }$, and in fact the above metric
coincides with the traditional lifted geodesic metric on the covering space $%
X_{\varepsilon }$ (Proposition 23, \cite{PW}). The definition using (\ref%
{mdeff}) is very useful for our purposes, but since we will need the lifted
geodesic metric for arbitrary covering spaces, we will recall the definition
now. Given a covering map $\phi :X\rightarrow Y$, where $Y$ is a geodesic
space and $X$ is connected, the lifted geodesic metric on $X$ is defined by $%
d(x,y)=\inf \{L(\phi \circ c):c$ is a path joining $x$ and $y\}$. As pointed
out in \cite{PW}, geodesic metrics are uniquely determined by their local
values, and in particular the lifted geodesic metric is uniquely determined
by the fact that $\phi $ is a local isometry.

There is a mapping from fixed-endpoint homotopy classes of continuous paths
to $\varepsilon $-homotopy classes of $\varepsilon $-chains defined as
follows: For any continuous path $c:[0,1]\rightarrow X$, choose $%
0=t_{0}<\cdot \cdot \cdot <t_{n}=1$ fine enough that every image $%
c([t_{i},t_{i+1}])$ is contained in the open ball $B(c(t_{i}),\varepsilon )$%
. Then the chain $\{c(t_{0}),...,c(t_{n})\}$ is called a subdivision $%
\varepsilon $-chain of $c$. Setting $\Lambda([c]):=[c(t_{0}),...,c(t_{n})]_{%
\varepsilon }$ produces a well-defined function that is length
non-increasing in the sense that $\left\vert \Lambda ([c])\right\vert \leq
\left\vert \lbrack c]\right\vert :=\inf \{L(d):d\in \lbrack c]\}$.
Restricting $\Lambda $ to the fundamental group at any base point yields a
homomorphism $\pi _{1}(X)\rightarrow \pi _{\varepsilon }(X)$ that we will
also refer to as $\Lambda $. When $X$ is geodesic, $\Lambda $ is surjective
since the successive points of an $\varepsilon $-loop $\lambda$ may be
joined by geodesics to obtain a path loop whose class goes to $[\alpha
]_{\varepsilon }$. The kernel of $\Lambda $ is precisely described by
Corollary \ref{lker}. Variations of $\Lambda $ and their applications to the
fundamental group and universal covers are further examined by the second
author in \cite{W}.

A partial inverse operation to $\Lambda $ is given by the following notion:
Let $\alpha :=\{x_{0},...,x_{n}\}$ be an $\varepsilon $-chain in a metric
space $X$, where $\varepsilon >0$. A \textit{stringing} of $\alpha $
consists of a path $\widehat{\alpha }$ formed by concatenating paths $%
\gamma_{i}$ from $x_{i}$ to $x_{i+1}$ where each path $\gamma_{i}$ lies
entirely in $B(x_{i},\varepsilon )$. If each $\gamma _{i}$ is a geodesic
then we call $\widehat{\alpha }$ a \textit{chording }of $\alpha $. Note that
by \textquotedblleft geodesic\textquotedblright\ in this paper we mean an
arclength-parameterized path whose length is equal to the distance between
its endpoints, and not a locally minimizing path as is the more common
meaning in Riemannian geometry. We will need the following two basic results.

\begin{proposition}
\label{ender}If $\alpha $ is an $\varepsilon $-chain in a chain connected
metric space $X$ then the unique lift of any stringing $\widehat{\alpha }$
starting at the basepoint $[\ast ]_{\varepsilon }$ in $X_{\varepsilon }$ has 
$[\alpha ]_{\varepsilon }$ as its endpoint.
\end{proposition}

\begin{corollary}
\label{homimp}If $\alpha $ and $\beta $ are $\varepsilon $-chains in a chain
connected metric space $X$ such that there exist stringings $\widehat{\alpha 
}$ and $\widehat{\beta }$ that are path homotopic then $\alpha $ and $\beta $
are $\varepsilon $-homotopic.
\end{corollary}

We also need some basic technical results. The first of these quantifies the
idea that \textquotedblleft uniformly close\textquotedblright\ $\varepsilon $%
-chains are $\varepsilon $-homotopic. Of course \textquotedblleft
close\textquotedblright\ depends on $\varepsilon $. Given $\alpha
=\{x_{0},...,x_{n}\}$ and $\beta =\{y_{0},...,y_{n}\}$ with $x_{i},y_{i}\in
X $, define $\Delta (\alpha ,\beta ):=\underset{i}{\max }\{d(x_{i},y_{i})\}$%
. For any $\varepsilon >0$, if $\alpha $ is an $\varepsilon $-chain we
define $E_{\varepsilon }(\alpha ):=\underset{i}{\min }\{%
\varepsilon-d(x_{i},x_{i+1})\}>0$. When no confusion will result we will
eliminate the $\varepsilon $ subscript.

\begin{proposition}
\label{close}Let $X$ be a metric space and $\varepsilon >0$. If $%
\alpha=\{x_{0},...,x_{n}\}$ is an $\varepsilon $-chain and $%
\beta=\{x_{0}=y_{0},...,y_{n}=x_{n}\}$ is such that $\Delta (\alpha ,\beta
)< \frac{E(\alpha )}{2}$ then $\beta $ is an $\varepsilon $-chain that is $%
\varepsilon $-homotopic to $\alpha $.
\end{proposition}

The reader will likely have noticed that the previous proposition requires
that the chains in question have the same number of points. The next lemma
shows that this is not really an issue. It is useful in many ways--for
example to find \textquotedblleft convergent subsequences\textquotedblright
of classes of chains, much like a discrete version of Ascoli's Theorem.

\begin{lemma}
\label{lestl}Let $L,\varepsilon >0$ and $\alpha $ be an $\varepsilon $-chain
in a metric space $X$ with $L(\alpha )\leq L$. Then there is some $%
\alpha^{\prime }\in \lbrack \alpha ]_{\varepsilon }$ such that $L(\alpha
^{\prime})\leq L(\alpha )$ and $\nu (\alpha ^{\prime })=\left\lfloor \frac{2L%
}{\varepsilon }+1\right\rfloor $.
\end{lemma}

For any $\delta \geq \varepsilon >0$, every $\varepsilon $-chain
(respectively $\varepsilon $-homotopy) is also a $\delta $-chain
(respectively $\delta $-homotopy) and there is a well-defined mapping $%
\phi_{\delta \varepsilon }:X_{\varepsilon }\rightarrow X_{\delta }$ given by 
$\phi _{\delta \varepsilon }([\alpha ]_{\varepsilon })=[\alpha ]_{\delta }$.
When $X$ is geodesic, this mapping is also a regular covering map and local
isometry, though for non-geodesic metric spaces it may not be surjective.
Restricting the map $\phi_{\delta \varepsilon}$ to the group $%
\pi_{\delta}(X) $ induces a homomorphism $\theta_{\delta
\varepsilon}:\pi_{\varepsilon}(X) \rightarrow \pi_{\delta}(X)$, which is
injective (respectively, surjective) if and only if $\phi_{\delta
\varepsilon}$ is. Thus, for a geodesic space $X$, one obtains parameterized
collections of covering spaces $\{X_{\varepsilon}\}_{\varepsilon > 0}$ and
their corresponding deck groups $\{\pi_{\varepsilon}(X)\}_{\varepsilon > 0}$%
, which actually form inverse systems (see \cite{BPUU}) via the surjective
bonding maps/homomorphisms $\phi_{\delta \varepsilon}$ and $\theta_{\delta
\varepsilon}$, respectively.

A number $\varepsilon >0$ is called a \textit{homotopy critical value }for $%
X $ if there is an $\varepsilon $-loop $\alpha $ based at $\ast $ such that $%
\alpha $ is not \textit{$\varepsilon $-null} (i.e. $\varepsilon $-homotopic
to the trivial chain) but is $\delta $-null for all $\delta >\varepsilon $.
We have the following essential connection between homotopy critical values
and $\varepsilon $-covers:

\begin{lemma}
\label{nocrit}If $X$ is a geodesic space then the covering map $%
\phi_{\varepsilon \delta }:X_{\delta }\rightarrow X_{\varepsilon }$ is
injective if and only if there are no homotopy critical values $\sigma $
with $\delta \leq \sigma <\varepsilon $.
\end{lemma}

\begin{corollary}
\label{three}If $\lambda $ is an $\varepsilon $-loop in a geodesic space $X$
of length less than $3\varepsilon $ then $\lambda $ is $\varepsilon $-null.
\end{corollary}

If $\alpha =\{x_{0},..,x_{n}\}$ is a chain in a geodesic space $X$ then a
refinement of $\alpha $ consists of a chain $\beta $ formed by inserting
between each $x_{i}$ and $x_{i+1}$ some subdivision chain of a geodesic
joining $x_{i}$ and $x_{i+1}$. If $\beta $ is an $\varepsilon $-chain we
will call $\beta $ an $\varepsilon $-refinement of $\alpha $. Note that if $%
\alpha $ is an $\varepsilon $-chain, then any $\varepsilon $-refinement of $%
\alpha $ is $\varepsilon $-homotopic to $\alpha $, and, hence, any two $%
\varepsilon $-refinements of $\alpha $ are $\varepsilon $-homotopic. A
special case is a \textit{midpoint refinement}, which simply uses a midpoint
between pairs of points. Of course refinements always exist in geodesic
spaces, but not in general metric spaces. Refinements are an important tool
when working with convergence questions, since the property of being an $%
\varepsilon $-chain is an \textquotedblleft open
condition\textquotedblright\ and may not be preserved when passing from a
sequence of chains to its pointwise limit.

\begin{definition}
If $X$ is a metric space and $\varepsilon >0$, an $\varepsilon $-loop of the
form $\lambda =\alpha \ast \tau \ast \overline{\alpha }$, where $\nu (\tau)=3
$, will be called $\varepsilon $-small. Note that this notation includes the
case when $\alpha $ consists of a single point--i.e. $\lambda =\tau $.
\end{definition}

The next proposition was established in \cite{PW} with the assumption that $%
\varepsilon <\delta $, but the same proof works for $\varepsilon =\delta $,
and we will need that case in the present paper. In essence it
\textquotedblleft translates\textquotedblright\ a homotopy into a product of
small loops.

\begin{proposition}
\label{deltaimp}Let $X$ be a geodesic space and $0<\varepsilon \leq \delta $%
. Suppose $\alpha ,\beta $ are $\varepsilon $-chains and $\left\langle
\alpha =\gamma _{0},...,\gamma _{n}=\beta \right\rangle $ is a $\delta $%
-homotopy. Then $[\beta ]_{\varepsilon }=[\lambda _{1}\ast \cdot \cdot \cdot
\ast \lambda _{r}\ast \alpha \ast \lambda _{r+1}\ast \cdot \cdot \cdot \ast
\lambda _{n}]_{\varepsilon }$, where each $\lambda _{i}$ is an $\varepsilon $%
-refinement of a $\delta $-small loop.
\end{proposition}

An $\varepsilon $\textit{-triad} in a geodesic space $X$ is a triple $%
T:=\{x_{0},x_{1},x_{2}\}$ such that $d(x_{i},x_{j})=\varepsilon $ for all $%
i\neq j$; when $\varepsilon $ is not specified we will simply refer to a
triad. We denote by $\alpha _{T}$ the loop $\{x_{0},x_{1},x_{2},x_{0}\}$. We
say that $T$ is \textit{essential} if some (equivalently any) $\varepsilon $%
-refinement of $\alpha _{T}$ is not $\varepsilon $-null. The equivalence of
\textquotedblleft some\textquotedblright\ and \textquotedblleft
any\textquotedblright\ in the preceding definition -- as well as the useful
fact that the $\varepsilon $-refinement of $\alpha _{T}$ may always be taken
to be a midpoint refinement -- follow from Proposition 37 in \cite{PW}.
Essential $\varepsilon $-triads $T_{1}$ and $T_{2}$ are defined to be 
\textit{equivalent} if some (equivalently any) $\varepsilon $-refinement of $%
\alpha _{T_{1}}$ is freely $\varepsilon $-homotopic to an $\varepsilon $%
-refinement of either $\alpha _{T_{1}}$ or $\overline{\alpha _{T_{1}}}$, and
again it suffices to consider only midpoint refinements. See \cite{PW} for
the definition of \textquotedblleft free $\varepsilon $-homotopy%
\textquotedblright , which is analogous to the classical meaning for paths.

The image of a closed loop of length $3\varepsilon $ with the property that
any $\varepsilon $-loop along this path loop is not $\varepsilon $-null is
called an \textit{essential }$\varepsilon $\textit{-circle}. If one connects
the points of an essential $\varepsilon $-triad with minimal geodesics to
form a parameterized loop, the resulting loop forms an essential $%
\varepsilon $-circle. Conversely, if one subdivides an essential $%
\varepsilon $-circle into three equal segments, then the endpoints of those
segments are an essential $\varepsilon $-triad (Proposition 37 and Corollary
41, \cite{PW}). Equivalence of the underlying triads is used to define
equivalence of essential circles. Note that equivalent essential circles may
not be freely path homotopic due to \textquotedblleft small
holes\textquotedblright\ that block traditional homotopies but not $%
\varepsilon $-homotopies. It should also be noted that while essential
circles are necessarily the images of non-null, closed geodesics that are
shortest in their respective homotopy classes, they have an even stronger
property: an essential circle is metrically embedded in the sense that its
metric as a subspace of $X$ is the same as the intrinsic metric of the
circle (Theorem 39, \cite{PW}). There are examples of closed geodesics in
compact geodesic spaces that are not essential circles (Example 44, \cite{PW}%
).

The set of homotopy critical values of a compact geodesic space make up
what is called the \textit{homotopy critical spectrum} of $X$, and these
values indicate precisely when the equivalence type of the $\varepsilon $%
-covers changes as $\varepsilon $ decreases to $0$. For example, given a
standard geodesic (i.e. Riemannian) circle of circumference $a>0$, the $%
\varepsilon $-covers are all isometries when $\varepsilon >\frac{a}{3}$, but
are the standard universal cover when $\varepsilon \leq \frac{a}{3}$; that
is, the homotopy critical spectrum of this circle is $\{\frac{a}{3}\}$. Equivalently,
any $\frac{a}{3}$-loop that traverses the circle once in either direction is not
$\frac{a}{3}$-null, but it is $\delta$-null homotopic for all $\delta > \frac{a}{3}$.
In general one may imagine $\varepsilon $ as decreasing from the diameter of
the space towards $0$, as the $\varepsilon $-covers \textquotedblleft
unravel\textquotedblright\ more of the topology of the space at each
homotopy critical value. If the space is semilocally simply connected, then
the process stabilizes with the universal cover, and the smallest homotopy
critical value is $\frac{1}{3}$ of the $1$-systole of the space (smallest
non-null closed geodesic--Corollary 43, \cite{PW}). If the space is not
semilocally simply connected, then the \textquotedblleft
unrolling\textquotedblright\ process may never end, although by using an
inverse limit one will obtain what Berestovskii-Plaut call the \textit{%
uniform universal cover} as defined in \cite{BPUU}.

Sormani and Wei defined the \textit{covering spectrum} of a compact geodesic
space to be the set of all $\delta >0$ such that $\widetilde{X}^{\delta
}\neq \widetilde{X}^{\delta ^{\prime }}$ for all $\delta ^{\prime }>\delta $%
; hence, the covering spectrum also indicates where the $\delta $-covers
change equivalence type. As was noted in the introduction, while the
homotopy critical spectrum may be defined for more general metric spaces, it
follows from the definitions and Corollary \ref{sw} below that the homotopy critical spectrum of
a compact geodesic space differs from its covering spectrum only by a
constant multiple of $\frac{2}{3}$. For compact geodesic spaces, the
homotopy critical values are discrete and bounded above in $(0,\infty )$; in
fact, if $\mathcal{X}$ is a precompact collection of compact geodesic
spaces, there is a uniform upper bound on the number of critical values in
any positive interval $[a,b]$ for any $X\in \mathcal{X}$. This was first
proved by Sormani-Wei when the spaces have universal covers (Corollary 7.7, 
\cite{SW2} -- though they prove discreteness directly without this assumption)
and later without this assumption and with a different proof by Plaut-Wilkins 
(Theorem 11, \cite{PW}).

The connection between essential circles/triads and the homotopy critical
spectrum was established by the authors in Theorem 6 of \cite{PW}: $%
\varepsilon $ is a homotopy critical value of a compact geodesic space $X$
if and only if $X$ contains an essential $\varepsilon $-circle. The \textit{%
multiplicity of a critical value $\varepsilon $} is the number of
equivalence classes of essential $\varepsilon $-circles, which is always
finite when $X$ is compact. Furthermore, Sormani and Wei showed that if
compact geodesic spaces $X_{i}$ converge to the compact space $X$, then the
covering (hence, homotopy critical) spectra of the spaces $X_{i}$ converge
in the Hausdorff sense in $\mathbb{R}$ to the spectrum of $X$ (Theorem 8.4, 
\cite{SW2}), which means that essential circles in the limit $X$ can only arise
as ``limits'' of essential circles in the spaces $X_{i}$. With this in mind, one can
see that the fact that $\varepsilon $-covers are not closed with respect to 
Gromov-Hausdorff convergence is related to the behavior of both the covering 
spaces, themselves, and the homotopy critical spectra, as the next example shows.

\begin{example}
\label{circle ex} Suppose $X_{i}$ is a geodesic circle of circumference $a-%
\frac{1}{i}$, so $X_{i}\rightarrow X$, where $X$ is the circle of
circumference $a$.  The homotopy critical spectra of $X_{i}$ and $X$, respectively,
are $\left\{ \frac{a}{3}-\frac{1}{3i}\right\}$ and $\left\{ \frac{a}{3}\right\}$. If we set 
$\varepsilon = \frac{a}{3}$, then $(X_{i})_{\varepsilon}=X$, while $X_{\varepsilon}
=\mathbb{R}$. Of course in this case $(X_{i})_{\frac{a}{3}}$ \textit{does} converge 
to an $\varepsilon$-cover of $X$, but it is not $X_{a/3}$. The covers $(X_{i})_{a/3}$
converge to $X_{\tau}$ for any $\tau >\frac{a}{3}$.  On the other hand, for
any $\varepsilon \neq \frac{a}{3}$, it does hold that $(X_{i})_{\varepsilon} \rightarrow
X_{\varepsilon}$.
\end{example}

To see how it is possible for $(X_{i})_{\varepsilon }$ to not converge to
any $\varepsilon $-cover at all, it will be helpful to illustrate the notion
of multiplicity of a critical value $\varepsilon $ -- which by definition is
the number of equivalence classes of essential $\varepsilon $-circles.
Suppose that $Y$ denotes the flat torus obtained by identifying the sides of
a rectangle of dimensions $0<3a\leq 3b$. When $a<b$ , $a$ and $b$ are
distinct homotopy critical values: $Y_{\varepsilon }=Y$ for $\varepsilon >b$%
, $Y_{\varepsilon }$ is a flat metric cylinder over a circle of length $3a$
for $a<\varepsilon \leq b$, and $Y_{\varepsilon }$ is the plane for $%
\varepsilon \leq a$. Each critical value $a$, $b$ has multiplicity $1$
because there is one class of essential circles corresponding to each value.
When $a=b$, however, the torus unrolls immediately into the plane at $%
\varepsilon =a$; that is, the $\varepsilon $-covers \textquotedblleft
skip\textquotedblright\ the cylinder. In this case, $a$ is a homotopy
critical value of multiplicity $2$, since both topological holes are of the
same size and are detected by the $\varepsilon $-covers simultaneously.

\begin{example}
\label{torus ex} Take a sequence of tori $T_{i}$ obtained from $(1-\frac{1}{i%
})\times (1+ \frac{1}{i})$-rectangles. Then each $T_{i}$ has homotopy
critical values $\frac{1}{3 }-\frac{1}{3i}$ and $\frac{1}{3}+\frac{1}{3i}$,
each with single multiplicity, while the limiting torus has a single
critical value $\frac{1}{3}$ with multiplicity $2$. In fact, $(T_{i})_{\frac{%
1}{3}}$ is a cylinder for all $i$, with limit $Y$ a cylinder of
circumference $1$, which as observed above is not an $\varepsilon $-cover of
the limiting torus.
\end{example}

What happens in Example \ref{torus ex} is that distinct homotopy critical
values merge in the limit to a single homotopy critical value with
multiplicity greater than $1$. Theorem \ref{closed} formally describes how
this phenomenon occurs; by extending $\varepsilon $-covers to the notion of
circle covers, we can tease apart the multiplicity to find the
\textquotedblleft missing\textquotedblright\ intermediate covers. We
conclude this section with an example illustrating what can go wrong in
Theorem \ref{closed} when there is no positive lower bound on the size of
the essential circles.

\begin{example}
\label{zero limit ex} Consider the torus $T_{i}=S_{3/i}^{1}\times S_{1}^{1}$
formed by circles of circumference $\frac{3}{i}$ and $1$. Then $%
T_{i}\rightarrow S_{1}^{1}$. If we choose $\varepsilon _{i}$ so that $%
\varepsilon _{i}<\frac{1}{i}$ and $\varepsilon _{i}\rightarrow 0$, then each 
$(T_{i})_{\varepsilon _{i}}$ for $i\geq 3$ is the universal cover $\mathbb{R}%
^{2}$, but of course $\mathbb{R}^{2}$ is not any kind of cover of $S^{1}$.
This particular example satisfies the assumptions of Theorem 1.1, \cite{EW},
in which Ennis and Wei showed the following: Suppose $X_{i}\rightarrow X$
are all compact geodesic spaces having (categorical, possibly not simply
connected) universal covers. The latter assumption is equivalent to the
homotopy critical spectra having positive lower bounds $\varepsilon
_{i},\varepsilon $, and in fact the universal covers are $%
(X_{i})_{\varepsilon _{i}}$ and $X_{\varepsilon }$, respectively. Theorem
1.1, \cite{EW} says that if one additionally assumes that the spaces $X_{i}$
have dimension uniformly bounded above and the spaces $X_{\varepsilon _{i}}$
pointed Gromov-Hausdorff converge to a space $\overline{X}$, then there is a
subgroup $H\subset Iso(\bar{X})$ such that $\bar{X}/H$ is the universal
cover of $X$. Note that in the case of collapse, the subgroup $H$ need not
be discrete.

On the other hand, if $H$ denotes the geodesic Hawaiian Earring, then as $%
\varepsilon \rightarrow 0$ the $\varepsilon $-covers $H_{\varepsilon }$ of $H
$ contain graphs of higher and higher valency (see \cite{BPUU} for more
details). Thus, if $\inf \{\varepsilon_i\} = 0$, then no subsequence of $%
(X_i)_{\varepsilon_i}$ can converge.
\end{example}

\section{Covering Maps and Quotients}

We begin by recalling some results concerning quotients of metric spaces. In
this paper, all actions are by isometries and are discrete in the sense of 
\cite{PQ}. Discreteness of an action is a uniform version of properly
discontinuous, which is implied by the following property for any $G$ acting
by isometries on a metric space $Y$ : There exists some $\varepsilon >0$
such that for all $y\in Y$ and non-trivial $g\in G$, $d(y,g(y))\geq
\varepsilon $. Note that in the case where $Y$ is the $\varepsilon $-cover
of a metric space $X$, $\pi _{\varepsilon }(X)$ acts discretely on $%
X_{\varepsilon }$, since the restriction of $\phi _{\varepsilon }$ to any $%
B([\alpha ]_{\varepsilon },\frac{\varepsilon }{2})\subset X_{\varepsilon }$
is an isometry onto its image in $X$ (\cite{PW}). When $Y$ is geodesic, then
the quotient metric on $Y/G$ (cf. \cite{PS}) is the uniquely determined
geodesic metric such that the quotient mapping is a local isometry (\cite{PW}%
). Combining this observation with Proposition 28 in \cite{PQ} we obtain the
following:

\begin{proposition}
\label{quote}Suppose that $X$ is a geodesic space, $G$ acts discretely by
isometries on $X$, and $H$ is a normal subgroup of $G$. Then $G/H$ acts
discretely by isometries on $X/H$ via $gH(Hx)=g(x)H$ and the mapping $%
Gx\mapsto (G/H)(Hx)$ an isometry from $X/G$ to $(X/H)/(G/H)$.
\end{proposition}

Among basic applications we have that for any geodesic space $X$ and $%
0<\delta <\varepsilon $, the covering map $\phi
_{\varepsilon}:X_{\varepsilon }\rightarrow X$ is isometrically equivalent to
the induced mapping $\zeta :X_{\delta }/\ker \theta _{\delta \varepsilon
}\rightarrow X$. Here $\ker \theta _{\delta \varepsilon }$ acts discretely
and isometrically as a normal subgroup of $\pi _{\delta }(X)$ with $%
X=X_{\delta}/\pi _{\delta }(X)$, and $\zeta $ is the unique covering map
such that $\zeta \circ \pi =\phi _{\delta }$, while $\pi :X_{\delta
}\rightarrow X_{\delta }/\ker \theta _{\delta \varepsilon }$ is the quotient
mapping.

We know already from the results of \cite{BPUU} that if $X$ is a compact
metric space, $Y$ is connected, and $f:Y\rightarrow X$ is a covering map
then for small enough $\varepsilon >0$ there is a covering map $%
g:X_{\varepsilon }\rightarrow Y$. The next proposition refines this
statement when $X$ is geodesic.

\begin{proposition}
\label{lebesgue}Let $X$ be a compact geodesic space and suppose that $%
f:Y\rightarrow X$ is a covering map, where $Y$ is connected. Suppose that $%
\varepsilon >0$ is at most $\frac{2}{3}$ of a Lebesgue number for a covering
of $X$ by open sets evenly covered by $f$. Then there is a covering map $%
g:X_{\varepsilon }\rightarrow Y$ such that $\phi _{\varepsilon }=f\circ g$.
\end{proposition}

\begin{proof}
Choose a basepoint $\ast $ in $Y$ such that $f(\ast )=\ast $ and define $%
g([\alpha ]_{\varepsilon })$ to be the endpoint lift of some stringing $%
\widehat{\alpha }$ starting at $\ast $ in $Y$. We need to check that $g$ is
well-defined. By iteration, it suffices to prove the following: If $%
\alpha:=\{x_{0},...,x_{n}\}$ and $\alpha ^{\prime }$ is an $\varepsilon $%
-chain $\{x_{0},...,x_{i},x,x_{i+1},...,x_{n}\}$, then for any stringings $%
\widehat{\alpha }$ and $\widehat{\alpha ^{\prime }}$, the endpoints of the
lifts of $\widehat{\alpha }$ and $\widehat{\alpha ^{\prime }}$ starting at $%
\ast $ are the same. Let $\left\{ \gamma _{i}\right\} $ and $\left\{
\gamma_{i}^{\prime }\right\} $ be geodesics joining $x_{i},x_{i+1}$ and $%
\beta_{1},\beta _{2}$ be geodesics from $x_{i}$ to $x$ and $x$ to $x_{i+1}$,
respectively. Note that each of the loops formed by each $\gamma _{i}$ and
the reversal of $\gamma _{i}^{\prime }$ has diameter smaller than $%
\varepsilon $ and thus lifts as a loop in $Y$. Moreover, the triangle formed
by the geodesics $\beta _{1}$, $\beta _{2}$, and either $\gamma _{i}$ or $%
\gamma _{i}^{\prime }$ has diameter less than $\frac{3\varepsilon }{2}$ and
hence lifts as a loop. This proves that $g$ is well-defined. It now follows
from a standard result in topology that $g$ is a covering map (\cite{Mu}).
(This result is stated with the additional assumption that all maps are
continuous, but that assumption is superfluous because the third mapping - $%
g $ in this case - is locally a homeomorphism.)
\end{proof}

\vspace{.1 in} The next theorem shows that any covering space of a compact
geodesic space can be obtained in a particularly natural way as a quotient
of the space $X_{\varepsilon }$ obtained in Proposition \ref{lebesgue}.

\begin{theorem}
\label{mcov}Let $X$ be a compact geodesic space and suppose that $%
f:Y\rightarrow X$ is a covering map, where $Y$ is connected and has the
lifted geodesic metric from $X$. Let $\varepsilon >0$ be such that there is
a covering map $g:X_{\varepsilon }\rightarrow Y$ with $\phi
_{\varepsilon}=f\circ g$. Define a subgroup $K$ of $\pi_{\varepsilon}(X)$ by 
\begin{equation*}
K:=\{[\lambda ]_{\varepsilon }:\exists \text{ }\kappa \in \lbrack \lambda
]_{\varepsilon }\text{ such that some stringing }\widehat{\kappa }\text{ of }%
\kappa \text{ lifts as a loop in }Y\}\text{.}
\end{equation*}
Then

\begin{enumerate}
\item[\emph{1.}] $K=\{[\lambda ]_{\varepsilon }:$ for all $\kappa \in
\lbrack \lambda ]_{\varepsilon }$, every stringing $\widehat{\kappa }$ of $%
\kappa $ lifts as a loop to $Y\}$.

\item[\emph{2.}] There is a covering equivalence $\phi :Y\rightarrow
X_{\varepsilon }/K$ such that $\pi =\phi \circ g$, where $\pi:X_{\varepsilon
}\rightarrow X_{\varepsilon }/K$ is the quotient map.

\item[\emph{3.}] $K$ is a normal subgroup if and only if $f$ is a regular
covering map.
\end{enumerate}
\end{theorem}

\begin{proof}
The first part follows from Proposition \ref{ender}. Define $\phi (y):=\pi(x)
$, where $x\in g^{-1}(y)$. To see why $\phi $ is well-defined, suppose that $%
g([\alpha ]_{\varepsilon })=g([\beta ]_{\varepsilon })=y$. Since the lifts $%
\widetilde{\alpha }$ and $\widetilde{\beta }$ of $\widehat{\alpha }$ and $%
\widehat{\beta }$ to the basepoint in $X_{\varepsilon }$ have $%
[\alpha]_{\varepsilon }$ and $[\beta ]_{\varepsilon }$ as their endpoints
(Proposition \ref{ender}), $g\circ \widetilde{\alpha }$ and $g\circ 
\widetilde{\beta }$ both end in $y$. But these two curves are the unique
lifts of $\widehat{\alpha }$ and $\widehat{\beta }$ to $Y$, so they both end
in $y$. In other words if $\lambda :=\alpha \ast \overline{\beta }$, then $%
[\lambda ]_{\varepsilon }\in K$. Moreover, $[\alpha ]_{\varepsilon
}:=[\lambda ]_{\varepsilon }\ast \lbrack \beta ]_{\varepsilon }$, so $%
[\alpha ]_{\varepsilon }$ and $[\beta ]_{\varepsilon }$ lie in the same
orbit of $K$, and $\pi ([\alpha ]_{\varepsilon })=\pi ([\beta ]_{\varepsilon
})$. This shows that $\phi $ is well-defined, and clearly $\pi =\phi \circ g$%
.

For surjectivity, let $K[\alpha ]_{\varepsilon }\in X_{\varepsilon }/K$ and
define $y:=g([\alpha ]_{\varepsilon })$. We have 
\begin{equation*}
\phi (y)=\phi (g([\alpha ]_{\varepsilon }))=\pi ([\alpha
]_{\varepsilon})=K[\alpha ]_{\varepsilon }\text{.}
\end{equation*}
For injectivity, suppose that $y,z\in Y$ satisfy $\phi (y)=\phi (z)$. Then
for $[\alpha ]_{\varepsilon },[\beta ]_{\varepsilon }$ such that $%
g([\alpha]_{\varepsilon })=y$ and $g([\beta ]_{\varepsilon })=z$ we have
that $[\alpha ]_{\varepsilon },[\beta ]_{\varepsilon }$ lie in the same
orbit, so $[\alpha ]_{\varepsilon }\ast \lbrack \overline{\beta }%
]_{\varepsilon}=[\lambda ]_{\varepsilon }\in K$. Consequently, the lifts of
stringings $\widehat{\alpha }$ and $\widehat{\beta }$ to $Y$ at the
basepoint must have the same endpoint. But one lift ends in $y$ and the
other ends in $z$, so $y=z$. This shows that $\phi $ is a homeomorphism. The
natural covering map $\xi :X_{\varepsilon }/K\rightarrow X$ is defined by $%
\xi (K[\alpha]_{\varepsilon })=\phi _{\varepsilon }([\alpha ]_{\varepsilon })
$ and therefore $\xi \circ \phi =f$, showing that $\phi $ is a covering
equivalence.

If $K$ is normal then by Lemma 39 in \cite{PQ}, the covering map $%
\xi:X_{\varepsilon }/K\rightarrow X$ is a topological quotient map with
respect to a well-defined induced action of $\pi _{\varepsilon }(X)/K$,
defined by $[\lambda ]_{\varepsilon }K([\alpha ]_{\varepsilon
}K)=([\lambda]_{\varepsilon }([\alpha ])_{\varepsilon })K$. Since $\xi $ is
also a covering map, it follows from standard results in topology that $\xi $
is regular. Since $\phi $ is a covering equivalence, $f$ is also regular.
Conversely, suppose that $\xi $ is regular and let $H$ denote its group of
covering transformations. Define a function $h:\pi
_{\varepsilon}(X)\rightarrow H$ as follows: Given $[\alpha ]_{\varepsilon
}\in \pi_{\varepsilon }(X)$, let $y$ be the endpoint of the lift of some
chording $\widehat{\alpha }$ to $Y$ starting at the basepoint. Since $%
\phi_{\varepsilon }=g\circ f$, by uniqueness $y$ must also be the endpoint $%
f\circ c$, where $c$ is the lift of $\widehat{\alpha }$ to $X_{\varepsilon }$%
. By Proposition \ref{ender}, the endpoint of $c$ is just $%
[\alpha]_{\varepsilon }$, and hence $y=f([\alpha ]_{\varepsilon })$ depends
only on $[\alpha ]_{\varepsilon }$. That is, if we let $h([\alpha
]_{\varepsilon })$ be the (unique) $\mu \in H$ such that $\mu (\ast )=y$,
then $h$ is well-defined. Sorting through the definition shows that $h$ is a
homomorphism with kernel $K$, since $y$ is the basepoint exactly when $%
\widehat{\alpha }$ lifts as a loop.
\end{proof}

\begin{corollary}
\label{liftcor}Let $X$ be compact geodesic and $\varepsilon >0$. Then an $%
\varepsilon $-loop $\alpha $ is $\varepsilon $-null if and only if some
(equivalently any) stringing of $\alpha $ lifts to a loop in $X_{\varepsilon}
$.
\end{corollary}

\begin{proof}
We only note that $\alpha $ is $\varepsilon $-null if and only if $\beta
\ast \alpha \ast \overline{\beta }$ is $\varepsilon $-null for any $%
\varepsilon $-chain from the basepoint to the starting point of $\alpha $,
and the analogous statement also holds for null-homotopies of any stringing
of $\alpha $.
\end{proof}

\vspace{.1 in} The $\delta $-covering map defined by Sormani-Wei is obtained
using a construction of Spanier that provides for any open cover $\mathcal{U}
$ of a connected, locally path connected topological space $Z$ a covering
map $\phi _{\mathcal{U}}:W\rightarrow Z$. The covering map is characterized
by the fact that a path loop $c$ at the basepoint in $Z$ lifts as a loop to $%
W$ if and only if its homotopy equivalence class $[c]$ lies in the subgroup $%
S_{ \mathcal{U}}$ of $\pi _{1}(Z)$, which we will call the \textit{Spanier
Group of }$\mathcal{U}$, generated by all loops of the following form: $%
c\ast L\ast \overline{c}$, where $c$ is a path loop starting at the
basepoint and $L$ is a path loop lying entirely in some set in the open
cover $\mathcal{U}$ . For their construction, Sormani-Wei took $\mathcal{U}$
to be the cover of $X$ by open $\delta $-balls. We will denote the
corresponding Spanier Group by $S_{\delta }$.

\begin{corollary}
\label{sw}For any geodesic space $X$ and $\delta =\frac{3\varepsilon }{2}>0$%
, there is an equivalence of the covering maps (hence an isometry) $h: 
\widetilde{X}^{\delta }\rightarrow X_{\varepsilon }$.
\end{corollary}

\begin{proof}
An immediate consequence of the definition of $\widetilde{X}^{\delta }$ is
that all open $\delta $-balls are evenly covered by $\pi ^{\delta }$, and
hence we may take $\delta $ for the Lebesgue number in Proposition \ref%
{lebesgue}. That proposition gives a covering map $\psi :X_{\varepsilon
}\rightarrow \widetilde{X}^{\delta }$ and Theorem \ref{mcov} will finish the
proof if we can show that the group $K$ for $Y:=\widetilde{X}^{\delta }$ is
trivial. If $[\lambda ]_{\varepsilon }\in K$ then by definition some
chording $\widehat{\lambda }$ lifts as a loop in $\widetilde{X}^{\delta }$.
In other words, $\widehat{\lambda }$ is homotopic to a concatenation of
paths of the form $c_{i}\ast L_{i}\ast \overline{c_{i}}$ where $L_{i}$ is a
path loop that lies entirely in an open $\delta $-ball. According to
Corollary \ref{homimp} we need only show that any refinement $\varepsilon $%
-chain of any such path is $\varepsilon $-null. In other words, it is enough
to show the following: Any $\varepsilon $-refinement of a rectifiable loop $%
f $ in an open ball $B(x,\delta )$ is $\varepsilon $-null. Given such a loop 
$f $, the distance from points on $f$ to $x$ has a maximum $D<\delta $. Now
subdivide $f$ into segments $\sigma _{i}$ whose endpoints $x_{i},x_{i+1}$
satisfy $d(x_{i},x_{i+1})<\delta -D$. Then each of the $\varepsilon $-loops $%
\kappa _{i}:=\{x,x_{i},x_{i+1},x\}$ (where $x_{n}=x_{0}$ for the highest
index $n$) has length less than $D+D+\delta -D=D+\delta <2\delta
=3\varepsilon $. But then each $\kappa _{i}$, and hence $f$, is $\varepsilon 
$-null by Corollary \ref{three}.
\end{proof}

\begin{corollary}
\label{lker}If $X$ is a geodesic space, $\varepsilon >0$, and $\Lambda
:\pi_{1}(X)\rightarrow \pi _{\varepsilon }(X)$ is the natural homomorphism
defined in the second section, then $\ker \Lambda =S_{\frac{3\varepsilon }{2}%
}$.
\end{corollary}

\begin{proof}
We have $[c]\in \ker \Lambda $ if and only if some subdivision $\varepsilon $%
-chain $\alpha $ of $c$ is $\varepsilon $-null. Equivalently, by Corollary %
\ref{liftcor}, any stringing of $\alpha $ lifts to a loop in $X_{\varepsilon}
$. But since $X_{\varepsilon }=\widetilde{X}^{\delta }$ by Corollary \ref{sw}%
, this is equivalent to $[c]\in S_{\frac{3\varepsilon }{2}}$.
\end{proof}

\begin{definition}
\label{circle cover} Suppose $\mathcal{T}$ is a finite collection of
essential triads in a compact geodesic space and $\varepsilon >0$ is such
that each $T\in \mathcal{\ T}$ is a $\delta $-triad for some $\delta \geq
\varepsilon $. Define $K_{\varepsilon }(\mathcal{T)}$ to be the subgroup of $%
\pi _{\varepsilon }(X)$ generated by the collection $\Gamma _{\varepsilon }(%
\mathcal{T)}$ of all $[\alpha \ast T^{\prime }\ast \overline{\alpha }%
]_{\varepsilon }$, where $\alpha $ is an $\varepsilon $-chain starting at $%
\ast $ and $T^{\prime }$ is an $\varepsilon $-refinement of $T\in \mathcal{T}
$. Finally, we define $\phi_{\varepsilon }^{\mathcal{T}}:X_{\varepsilon
}/K_{\varepsilon }(\mathcal{T} )\rightarrow X$ by $\phi _{\varepsilon }^{%
\mathcal{T}}(K_{\varepsilon }( \mathcal{T)[\alpha ]}_{\varepsilon }):=\phi
_{\varepsilon }([\alpha]_{\varepsilon })$. We will call a covering map
equivalent to some $\phi_{\varepsilon }^{\mathcal{T}}$ a circle covering map
(including the case when $\mathcal{T}$ is empty, in which case we take $%
K_{\varepsilon }( \mathcal{T)}$ to be the trivial group, so $\phi
_{\varepsilon }^{\mathcal{T}}=\phi _{\varepsilon }$).
\end{definition}

\begin{remark}
\label{comments} First, note that $K_{\varepsilon }(\mathcal{T})$ is normal
in $\pi _{\varepsilon }(X)$, since any conjugate of $[\alpha \ast
T^{\prime}\ast \alpha ^{-1}]_{\varepsilon }$ has the same form. If $\mathcal{%
T}=\{T_{i}\}_{i=1}^{n}$, it is also easy to check that for any fixed choice
of $\varepsilon $-chains $\alpha _{i}$ from $\ast $ to $T_{i}$, $%
K_{\varepsilon}(\mathcal{T)}$ is the smallest normal subgroup containing the
finite set $\{[\alpha _{i}\ast T_{i}^{\prime }\ast \overline{\alpha _{i}}%
]_{\varepsilon }\}_{i=1}^{n}$. However, we do not know in general whether $%
K_{\varepsilon }(\mathcal{T})$ is finitely generated. Finally, a word of
caution. While $\delta $-refinements of an essential $\delta $ -triad $T$
are all $\delta $-homotopic, different $\varepsilon $-refinements of $T$
need not be $\varepsilon $-homotopic when $\delta $ is larger than $%
\varepsilon $. In fact, as simple geodesic graphs show, the vertices of $T$
may be joined by different geodesics that together form loops that are
always $\delta $-null but are not $\varepsilon $-null. That is, in general,
replacing a $\delta $-triad $T\in \mathcal{T}$ with a $\delta $-equivalent $%
\delta $-triad may change the group $K_{\varepsilon }(\mathcal{T})$. This
further emphasizes the essential dependence of $K_{\varepsilon }(\mathcal{T})
$ on not just the collection $\mathcal{T}$ but on the value of $\varepsilon $.
\end{remark}


\begin{notation}
When convenient we will denote $X$ by $X_{\infty }$; this makes sense since
any chain can be considered as an $\infty $-chain and every sequence of
chains in which one point is removed or added to get from one chain to the
next is an $\infty $-homotopy. Then every chain is $\infty $-homotopy
equivalent to the chain $\{x,y\}$ where $x$ and $y$ are its endpoints, and
hence the mapping $\phi _{\varepsilon }:X_{\varepsilon }\rightarrow X$ is
naturally identified with $\phi _{\infty \varepsilon }:X_{\varepsilon
}\rightarrow X_{\infty }$. This saves us from having to consider the mapping 
$\phi _{\varepsilon }$ as a special case in the statements that follow.
\end{notation}

\begin{proposition}
\label{delta}Let $\varepsilon >0$ be a homotopy critical value for a compact
geodesic space $X$. Then there is some $\delta >\varepsilon $ such that if $%
\{x_{0},x_{1},x_{2},x_{0}\}$ is $\delta $-small with a midpoint refinement $%
\alpha $ that is not $\varepsilon $-null then $\alpha $ is $\varepsilon $%
-homotopic to a midpoint refinement of an essential $\varepsilon $-triad.
\end{proposition}

\begin{proof}
If the statement were not true then there would exist $\left( \varepsilon + 
\frac{1}{i}\right) $-small loops $\{x_{i},y_{i},z_{i},x_{i}\}$ having
midpoint subdivision chains $\mu_{i}=%
\{x_{i},m_{i},y_{i},n_{i},z_{i},p_{i},x_{i}\}$ that are not $\varepsilon $%
-null but are not $\varepsilon $-homotopic to a midpoint refinement of an
essential $\varepsilon $-triad. By taking subsequences if necessary, we may
suppose that $\{x_{i},m_{i},y_{i},n_{i},z_{i},p_{i},x_{i}\}\rightarrow
\{x,m,y,n,z,p,x\}$, where $\{x,y,z,x\}$ is an $\varepsilon $-small loop
which is not $\varepsilon $-null. Hence $\{x,y,z\}$ must be an essential $%
\varepsilon $-triad. Moreover, for large $i$, $%
\{x_{i},m_{i},y_{i},n_{i},z_{i},p_{i},x_{i}\}$ is $\varepsilon $-homotopic
to a midpoint subdivision of $\{x,y,z,x\}$, a contradiction.
\end{proof}

\begin{definition}
\label{lollipop} In a geodesic space $X$, we will call a path (resp. $%
\varepsilon $-chain) of the form $k\ast c\ast \overline{k}$, where $k$ is a
path (resp. $\varepsilon $-chain), $\overline{k}$ is its reversal, and $c$
is an arclength parameterization of an essential circle (resp. a midpoint
refinement of an essential $\varepsilon$-triad), a lollipop (resp. $%
\varepsilon $-lollichain). If the path $k$ in the lollipop is locally length
minimizing (possibly not minimal) then we call the lollipop a geodesic
lollipop.
\end{definition}

Note that for fixed $\varepsilon $, the homomorphism $\Lambda
:\pi_{1}(X)\rightarrow \pi _{\varepsilon }(X)$ maps classes of lollipops
determined by essential $\varepsilon $-circles to $\varepsilon $%
-lollichains. In fact, if $c$ is an essential $\varepsilon $-circle
determined by an essential $\varepsilon $-triad $T$ then we can choose a
midpoint refinement $\beta $ of $\alpha _{T}$ such that each point of $\beta 
$ lies on $c$. By definition, $[c]$ is mapped via $\Lambda $ to $%
[\beta]_{\varepsilon }$. Conversely, if $\beta $ is a midpoint refinement of
an essential $\varepsilon $-triad $T$, we can define an essential $%
\varepsilon $-circle $c$ containing $\beta $ by joining the points of $\beta 
$ by geodesics. Then $\Lambda $ will map $[c]$ to $[\beta ]_{\varepsilon }$.
Anchoring the essential circles to the base point by adjoining paths $k$ and
choosing corresponding $\varepsilon $-chains $\kappa $ so that $%
[\kappa]_{\varepsilon }=\Lambda ([k])$ yields the full conclusion.

\begin{theorem}
\label{circles}Let $X$ be a compact geodesic space, $0<\varepsilon <\delta
\leq \infty $, and $\mathcal{T}$ any (possibly empty!) collection that
contains a representative for every essential $\tau $-triad with $%
\varepsilon \leq \tau <\delta $. Then $\ker \theta _{\delta \varepsilon
}=K_{\varepsilon }(\mathcal{T})$. Consequently, $\pi _{\delta
}(X)=\pi_{\varepsilon }(X)/K_{\varepsilon }(\mathcal{T})$, and the covering
map $\phi _{\delta \varepsilon }:X_{\varepsilon }\rightarrow X_{\delta }$ is
equivalent to the quotient covering map $\pi :X_{\varepsilon }\rightarrow
X_{\varepsilon }/K_{\varepsilon }(\mathcal{T})$.
\end{theorem}

\begin{proof}
First of all, note that the inequality $\tau <\delta $ shows that each
element of $K_{\varepsilon }(\mathcal{T})$ is $\delta $-null and hence $%
K_{\varepsilon }(\mathcal{T})\subset \ker \theta _{\delta \varepsilon }$.
For the opposite inclusion, let $[\lambda ]_{\varepsilon }\in \ker \theta
_{\delta \varepsilon }$, meaning that $\lambda $ is $\delta $-null. We will
start with the case when $\varepsilon $ is a homotopy critical value of $X$
and $\delta >\varepsilon $ is close enough to $\varepsilon $ that $\delta
<2\varepsilon $ and Proposition \ref{delta} is valid: whenever $%
\{x_{0},x_{1},x_{2},x_{0}\}$ is $\delta $-small with a midpoint refinement $%
\alpha $ that is not $\varepsilon $-null then $\alpha $ is $\varepsilon $%
-homotopic to a midpoint refinement of an essential $\varepsilon $-triad. By
Proposition \ref{deltaimp} $\lambda $ is $\varepsilon $-homotopic to a
product of midpoint refinements $\lambda _{i}$ of $\delta $-small loops.
Since Proposition \ref{delta} holds, each $\lambda _{i}$ is either $%
\varepsilon $-null or $\varepsilon $-homotopic to a non-null $\varepsilon $%
-lollichain. That is, $[\lambda ]_{\varepsilon }\in K_{\varepsilon }(%
\mathcal{T})$.

Next, observe that if there are no homotopy critical values $\tau $ with $%
\varepsilon \leq \tau <\delta $, then on the one hand $\mathcal{T}$ must be
empty, and on the other hand, $\theta _{\delta \varepsilon }$ is an
isomorphism so its kernel is trivial, and we are finished. Suppose now that
there is a single critical value $\tau $ between $\varepsilon $ and $\delta $%
, which, by the previous case, may be assumed to satisfy $\varepsilon \leq
\tau <\delta $. We may choose $\delta _{1}$ with $\tau <\delta _{1}<\delta $
satisfying the requirement of the special case proved in the first paragraph
to obtain that $\ker \theta _{\delta _{1}\tau }=K_{\tau }(\mathcal{T})$. Now
both $\theta _{\tau \varepsilon }$ and $\theta _{\delta \delta _{1}}$ are
isomorphisms, and $\theta _{\delta \varepsilon }=\theta _{\delta
\delta_{1}}\circ \theta _{\delta _{1}\tau }\circ \theta _{\tau \varepsilon }$
by definition. Therefore, 
\begin{equation*}
\ker \theta _{\delta \varepsilon }=\theta _{\tau \varepsilon }^{-1}(\ker
\theta _{\delta _{1}\tau }\mathcal{)=}\theta _{\tau
\varepsilon}^{-1}(K_{\tau }(\mathcal{T))=}K_{\varepsilon }(\mathcal{T})\text{%
.}
\end{equation*}

For the general case we have $\varepsilon \leq \varepsilon _{i}<\cdot \cdot
\cdot <\varepsilon _{j}<\delta :=\varepsilon _{j+1}$ where $%
\{\varepsilon_{i},...,\varepsilon _{j}\}$ is the set of all homotopy
critical values between $\varepsilon $ and $\delta $, which has at least two
elements. For $i\leq k\leq j$, let $\mathcal{T}_{k}$ be the set of all $%
\varepsilon _{k}$-triads in $\mathcal{T}$, so that $\mathcal{T=\cup T}_{k}$.
By the previous case we have for all $k$ that $\ker \theta _{\varepsilon
_{k+1}\varepsilon_{k}}=K_{\varepsilon _{k}}(\mathcal{T}_{k})$. If $[\lambda
]_{\varepsilon}\in \ker \theta _{\delta \varepsilon }=\ker \theta
_{\varepsilon_{j+1}\varepsilon }$ then $x:=\theta _{\varepsilon
_{j}\varepsilon}([\lambda ]_{\varepsilon })\in \ker \theta _{\varepsilon
_{j+1}\varepsilon_{j}}=K_{\varepsilon _{j}}(\mathcal{T}_{j})$, so we may
write $x$ as a finite product $\Pi _{r}[\beta _{r}^{j}]_{\varepsilon _{j}}$
with where each $\beta _{r}^{j}$ is an $\varepsilon _{j}$-lollichain made
from an $\varepsilon _{j}$-refinement of some element of $\mathcal{T}_{j}$.
Now let $\lambda _{r}^{j}$ be any $\varepsilon $-refinement of $\beta
_{r}^{j}$. By definition, $[\lambda _{r}^{j}]_{\varepsilon }\in K(\mathcal{T}%
)$ for all $r$, and $\theta _{\varepsilon _{j}\varepsilon
}([\lambda_{r}^{j}]_{\varepsilon })=[\beta _{r}^{j}]_{\varepsilon _{j}}$.
Therefore we may write $[\lambda ]_{\varepsilon }=[\lambda _{j}^{\prime
}]_{\varepsilon}\left( \Pi _{r}[\lambda _{r}^{j}]_{\varepsilon }\right) $
for some $[\lambda _{j}^{\prime }]_{\varepsilon }\in \ker \theta
_{\varepsilon \varepsilon _{j}} $. Since $[\lambda _{j}^{\prime
}]_{\varepsilon }\in \ker \theta _{\varepsilon \varepsilon _{j}}$ we may
repeat the same argument to write $[\lambda _{j}^{\prime }]_{\varepsilon }$
as a product of some $[\lambda _{j}^{\prime \prime }]_{\varepsilon }\in \ker
\theta _{\varepsilon\varepsilon _{j-1}}$ and a finite product of elements of 
$K(\mathcal{T)}$ (consisting of $\varepsilon $-lollichains formed using $%
\varepsilon $-refinements of elements of $\mathcal{T}_{j-1}$). After
finitely many iterations of this argument we obtain that $[\lambda
]_{\varepsilon }\in K( \mathcal{T})$.
\end{proof}

\vspace{.1 in} It is a straightforward extension of a theorem of E. Cartan
that in a compact semi-locally simply connected geodesic space, every path
contains a shortest path in its fixed-endpoint homotopy class, and that path
is a locally minimizing geodesic. Likewise, every path loop contains a
shortest element in its free homotopy class, and this curve is a closed
geodesic. (This is not true in general without semilocal simple
connectivity.) Consequently, in such spaces the fundamental group is
generated by homotopy classes of loops of the form $\alpha\ast c\ast \bar{%
\alpha} $, where $c$ is a closed geodesic that is shortest in its homotopy
class, and $\alpha $ is a locally minimizing geodesic. Now in the case $%
\delta =\infty $ in Theorem \ref{circles}, since we know from \cite{PW} that 
$\pi _{\varepsilon }(X)$ is finitely generated, we obtain that $%
\pi_{\varepsilon }(X)$ is generated by finitely many $\varepsilon $%
-lollichains. If $X$ is semilocally simply connected then for $\varepsilon $
small enough, $\pi _{\varepsilon }(X)$ is isomorphic to $\pi _{1}(X)$ via
the map $\Lambda:\pi_1(X) \rightarrow \pi_{\varepsilon}(X)$ (c.f. \cite{PW}
or \cite{W}). By the discussion following Definition \ref{lollipop},
applying $\Lambda^{-1}$ to a basis of $\varepsilon$-lollichains give a basis
of lollipops for $\pi _{1}(X)$, and we may replace any lollipop by a
geodesic lollipop, up to homotopy equivalence. We thus have shown:

\begin{theorem}
\label{newbase}Let $X$ be a compact geodesic space. Then $\pi
_{\varepsilon}(X)$ is either trivial or is generated by a finite collection $%
[\lambda_{1}]_{\varepsilon },...,[\lambda _{n}]_{\varepsilon }$, where each $%
\lambda_{i}$ is an $\varepsilon $-refinement of a $\delta $-lollichain with $%
\delta \geq \varepsilon $. In particular, if $X$ is semilocally simply
connected and not simply connected then $\pi _{1}(X)$ is generated by a
finite collection of equivalence classes of geodesic lollipops.
\end{theorem}

\noindent Recalling that the notion of essential circle is stronger than
just being a non-null, closed geodesic that is shortest in its homotopy
class, we see that Theorem \ref{newbase} is stronger than Cartan's result
even in the Riemannian case. In fact, Example 44 of \cite{PW} shows that
such closed geodesics need not be essential circles even when they are
shortest in their homotopy class in a Riemannian manifold.

We conjecture that there is always a collection of $\delta $-lollichains
that gives a generating set of $\pi _{\varepsilon }(X)$ having minimal
cardinality.

\begin{proposition}
\label{T}Let $X$ be a compact geodesic space, $0<\delta <\varepsilon $, and $%
\mathcal{T}$ be a collection of essential $\tau $-triads such that $\tau
\geq \varepsilon $ for each element of $\mathcal{T}$. Let $\mathcal{S}=%
\mathcal{T\cup T}^{\prime }$, where $\mathcal{T}^{\prime }$ consists of one
representative of each essential $\tau $-triad with $\delta \leq \tau
<\varepsilon $. Then

\begin{enumerate}
\item The covering map $\phi _{\varepsilon }^{\mathcal{T}}:X_{\varepsilon }^{%
\mathcal{T}}\rightarrow X$ is isometrically equivalent to $\phi _{\delta }^{%
\mathcal{S}}:X_{\delta }^{\mathcal{S}}\rightarrow X$ and

\item $\pi _{\varepsilon }^{\mathcal{T}}(X)$ is isomorphic to $\pi
_{\delta}^{\mathcal{S}}(X)$.
\end{enumerate}
\end{proposition}

\begin{proof}
We will use Proposition \ref{quote}. First observe that $K_{\delta }(%
\mathcal{T}^{\prime })$, as a normal subgroup of $\pi _{\delta }(X)$, is a
normal subgroup of $K_{\delta }(\mathcal{S)}$. By Theorem \ref{circles}, $%
\ker \theta _{\varepsilon \delta }=K_{\delta }(\mathcal{T}^{\prime })$ and
therefore we may identify the action of $\pi _{\varepsilon }(X)$ on $%
X_{\varepsilon }$ with the action of $K_{\delta }(\mathcal{T)}/K_{\delta }(%
\mathcal{T}^{\prime })$ on $X_{\delta }/K_{\delta }(\mathcal{T}%
^{\prime})=X_{\varepsilon }$. In order to apply Proposition \ref{quote} and
finish the proof of the first part, we need to show that $\theta
_{\varepsilon\delta }(K_{\delta }(\mathcal{S}))=K_{\varepsilon }(\mathcal{T)}
$. Since $\mathcal{T\subset S}$, $K_{\varepsilon }(\mathcal{T)=}\theta
_{\varepsilon\delta }(K_{\delta }(\mathcal{T}))\subset \theta _{\varepsilon
\delta}(K_{\delta }(\mathcal{S}))$ . On the other hand, let $%
[\lambda]_{\varepsilon }\in \theta _{\varepsilon \delta }(K_{\delta }(%
\mathcal{S}))$. By definition, $[\lambda ]_{\varepsilon }=[\lambda
_{1}]_{\varepsilon}\ast \cdot \cdot \cdot \ast \lbrack \lambda
_{k}]_{\varepsilon }$, where each $\lambda _{i}=\alpha _{i}\ast \beta
_{i}\ast \overline{\alpha _{i}}$, $\alpha _{i}$ is a $\delta $ -chain and
each $\beta _{i}$ is either in $\mathcal{T}$ or $\mathcal{T}^{\prime }$. But
if $\beta _{i}\in \mathcal{T}^{\prime }$ then $\beta _{i}$ is $\varepsilon $%
-null, so $[\lambda _{i}]_{\varepsilon }$ is trivial and we may therefore
eliminate $[\lambda_{i}]_{\varepsilon }$ from the product. The remaining
terms are all in $K_{\varepsilon }(\mathcal{T)}$.

For the second part, note that we have shown both $\ker \theta
_{\varepsilon\delta }=K_{\delta }(\mathcal{T}^{\prime })\subset K_{\delta }(%
\mathcal{S)}$ and $\theta _{\varepsilon \delta }(K_{\delta }(\mathcal{S}%
))=K_{\varepsilon}(\mathcal{T)}$. Therefore, from a basic theorem in algebra
we may conclude that $\pi _{\varepsilon }^{\mathcal{T}}(X)=\pi
_{\varepsilon}(X)/K_{\varepsilon }(\mathcal{T)}$ is isomorphic to $\pi
_{\delta}(X)/K_{\delta }(\mathcal{S)=\pi }_{\delta }^{\mathcal{S}}(X)$.
\end{proof}

\section{Gromov-Hausdorff Convergence}

\begin{definition}
\label{scaled}Suppose $f:X\rightarrow Y$ is a function between metric spaces
and $\sigma $ is a first degree polynomial with non-negative coefficients.
We say that $f$ is a $\sigma $-isometry if for all $x,y\in X$, $z\in Y$,

\begin{enumerate}
\item[\emph{1.}] $\left\vert d(x,y)-d(f(x),f(y))\right\vert \leq \sigma
(d(x,y))$ and

\item[\emph{2.}] $d(z,f(w))\leq \sigma (0)$ for some $w\in X$.
\end{enumerate}

\noindent We refer to the first condition as \textquotedblleft distortion at
most $\sigma $\textquotedblright.
\end{definition}

If $\sigma =0$ then a $\sigma $-isometry is an isometry, and if $%
\sigma=\varepsilon >0$ is constant then Definition \ref{scaled} agrees with
the notion of an $\varepsilon $-isometry given in \cite{BB}. If $X$ and $Y$
are compact of diameter at most $R$, then any $\sigma $-isometry is a $%
\sigma(R) $-isometry. In fact, if $X$ and $Y$ are compact, $\sigma $ is
constant, and $d$ denotes the Gromov-Hausdorff distance, then $%
d(X,Y)<2\sigma $ if there is a $\sigma $-isometry $f:X\rightarrow Y$, and
such a $\sigma $-isometry exists if $d(X,Y)<\frac{\sigma }{2}$ (Corollary
7.3.28, \cite{BB}). In other words, for purposes involving convergence of
compact spaces we might as well use constant functions $\sigma $. However,
our extended definition is needed to study the induced mapping $%
f_{\#}:X_{\delta}\rightarrow Y_{\varepsilon }$ since $X_{\sigma }$ and $%
Y_{\varepsilon }$ are not generally compact even when $X$ and $Y$ are.

\begin{remark}
\label{quasi}Recall that a quasi-isometry (c.f. \cite{BB}) is a map $%
f:X\rightarrow Y$ such that $f(X)$ is a $D$-net in $Y$ for some $D>0$ (i.e.
for every $y\in Y$ we have $d(y,f(x))<D$ for some $x\in X$) and $\frac{1}{%
\lambda }d(x,y)-C\leq d(f(x),f(y))\leq \lambda d(x,y)+C$ for all $x,y$ and
some constants $\lambda \geq 1$, $C>0$. If $\sigma (x)=mx+b$ then a $\sigma $%
-isometry is a quasi-isometry with $\lambda :=\frac{1+m}{1-m}$ and $C=b$,
with $\lambda \rightarrow 1$ as $m\rightarrow 0$. However, it is simpler for
our purposes to use Definition \ref{scaled}.
\end{remark}

\begin{proposition}
\label{ghpoly}Suppose that $(X_{i},x_{i})$, $(X,x)$ are proper geodesic
spaces and $f_{i}:X_{i}\rightarrow X$ is a basepoint preserving $\sigma _{i}$%
-isometry for all $i$, where $\sigma _{i}$ is a first degree polynomial with 
$\sigma _{i}\rightarrow 0$ pointwise. Then $(X_{i},x_{i})$ is pointed
Gromov-Hausdorff convergent to $(X,x)$.
\end{proposition}

\begin{proof}
It suffices to show that for any constants $0<\sigma <1<R$, there is a $%
\sigma $-isometry $g_{i}:B(x_{i},R)\rightarrow B(x,R)$ for all large $i$.
For large $i$, $\sigma _{i}(4R)<\frac{\sigma }{4}$, and the restriction of $%
f_{i}$ to $B(x_{i},2R)$ has distortion less than $\frac{\sigma }{4}$. In
particular, if $y\in B(x_{i},R)$ then $f_{i}(y)\in B(x,R+\frac{\sigma }{4})$%
. If $f_{i}(y)\in B(x,R)$, let $g_{i}(y):=f_{i}(y)$. Otherwise, since $X$ is
geodesic, there is some $u\in B(x,R)$ such that $d(f_{i}(y),u)<\frac{\sigma}{%
4}$ (i.e. $u$ is on a geodesic from $f_{i}(y)$ to $x$); in this case define $%
g_{i}(y):=u$. By the triangle inequality, the distortion of $g_{i}$ on $%
B(x_{i},R)$ is at most $\frac{3\sigma }{4}$.

To finish, we need only check condition 2 of the definition of $\sigma $%
-isometry for the constant $\sigma $. If $\sigma_{i}(t)=m_{i}t+b_{i}$, then
by our choice of $i$ (and $R>1$), we have $m_i,b_i < \frac{\sigma }{4}<\frac{%
1}{4}$. Let $z\in B(x,R)$; since $X$ is geodesic we may find $z^{\prime }\in
B(x,R-\frac{\sigma }{2})$ such that $d(z,z^{\prime })<\frac{\sigma }{2}$.
Since $\sigma _{i}(0)=b_{i}<\frac{\sigma }{4}$, there is some $w\in X_{i}$
such that $d(f_{i}(w),z^{\prime })<\frac{\sigma }{4}$. By the triangle
inequality, $f_{i}(w)\in B(x,R-\frac{\sigma }{4})$, and hence by definition
of $g_{i} $, $g_{i}(w)=f_{i}(w)$. Again by the triangle inequality, $%
d(g_{i}(w),z)\leq \frac{3\sigma }{4}<\sigma $, and we are left only to show
that $w\in B(x_{i},R)$. Set $d(w,x_{i}):=S$. Then 
\begin{equation*}
S \leq d(f_{i}(w),x)+\sigma _{i}(S) < R-\frac{\sigma }{4}+\sigma _{i}(S) <R-%
\frac{\sigma }{4}+\frac{S}{4}+\frac{\sigma }{4},
\end{equation*}
which implies $S<\frac{4}{3}R$. Since now $w\in B(x_{i},2R)$, we can use the
distortion of $\sigma _{i}$ to on this ball to improve this estimate and
complete the proof: 
\begin{equation*}
S<d(f_{i}(w),f_{i}(x_{i}))+\frac{\sigma }{4}=d(f_{i}(w),x)+\frac{\sigma }{4}%
<R-\frac{\sigma }{4}+\frac{\sigma }{4}=R\text{.}
\end{equation*}
\end{proof}

\vspace{.1 in} Let $f:X\rightarrow Y$ be a function between metric spaces.
We will extend the notion of the induced function $f_{\#}$ from \cite{BPUU}
as follows. For any metric space $Z$ and $\mu >0$, let $\overline{Z_{\mu }}$
consist of all $[\alpha ]_{\mu }$, where $\alpha $ is a $\mu $-chain (not
necessarily starting at $\ast $). For any chain $\alpha
=\{x_{0},\dots,x_{n}\}$ in $X$, we let $f(\alpha
):=\{f(x_{0}),\dots,f(x_{n})\}$. Note that $f(\alpha \ast \beta )=f(\alpha
)\ast f(\beta )$ whenever the first concatenation is defined. Suppose now
that for some fixed $\varepsilon ,\delta >0$ and all $x,y\in X$, if $%
d(x,y)<\varepsilon $ then $d(f(x),f(y))<\delta $. If $\alpha $ is an $%
\varepsilon $-chain in $X$ then $f(\alpha )$ is a $\delta $-chain in $Y$.
Moreover, if $\eta =\{\eta _{1},\dots,\eta _{n}\}$ is an $\varepsilon $%
-homotopy in $X$ then $f(\eta ):=\{f(\eta _{1}),\dots,f(\eta _{n})\}$ is a $%
\delta $-homotopy in $Y$. It follows that the mapping $f_{\#}:\overline{%
X_{\varepsilon }}\rightarrow \overline{Y_{\delta }}$ defined by $f_{\#}\left(%
\left[ \alpha \right] _{\varepsilon }\right)=[f(\alpha )]_{\delta }$ is
well-defined and satisfies 
\begin{equation}
f_{\#}([\alpha \ast \beta ]_{\varepsilon })=f_{\#}([\alpha
]_{\varepsilon})\ast f_{\#}([\beta ]_{\varepsilon })  \label{concat}
\end{equation}
whenever the first concatenation is defined.

The next technical proposition sorts through the ways in which $f_{\#}$
inherits the properties of a $\sigma $-isometry $f$. The statement is not
the most general possible--for example it is possible to consider
non-constant $\sigma $ and the first part doesn't depend on the full
distortion assumption--but the statement is complicated enough as it is and
we will not need more general statements for this paper.

\begin{proposition}
\label{serious}Let $X,Y$ be metric spaces, $0<\frac{\varepsilon }{2}%
<\omega<\delta <\varepsilon $. Suppose that $f:X\rightarrow Y$ is basepoint
preserving $\sigma $-isometry for some constant $\sigma $ with $0\leq \sigma
<\min \{\varepsilon -\delta ,\frac{\delta -\omega }{4}\}$. Then the induced
map $f_{\#}:\overline{X_{\delta }}\rightarrow \overline{Y_{\varepsilon }}$
is defined, and the following hold.

\begin{enumerate}
\item[\emph{1.}] For any $\delta $-chain $\alpha $ in $X$, 
\begin{equation*}
\left\vert f_{\#}([\alpha ]_{\delta })\right\vert \leq \left\vert \lbrack
\alpha ]_{\delta }\right\vert +\sigma \left( \frac{4\left\vert
[\alpha]_{\delta }\right\vert }{\varepsilon }+1\right).
\end{equation*}

\item[\emph{2.}] For any $\omega $-chain $\alpha $ in $X$, 
\begin{equation*}
\left\vert \lbrack f(\alpha )]_{\omega +\sigma }\right\vert \geq \left\vert
\lbrack \alpha ]_{\delta }\right\vert -\sigma \left( \frac{4}{\varepsilon }+ 
\frac{16\sigma }{\varepsilon ^{2}}\right) \left( \left\vert [\alpha]_{\omega
}\right\vert \right) -\sigma \left( \frac{4\sigma }{\varepsilon }+1\right).
\end{equation*}

\item[\emph{3.}] If $\beta $ is an $\left( \omega -3\sigma \right) $-chain
in $Y$ starting at $\ast $ then there exists some $\omega $-chain $\alpha $
in $X$ starting at $\ast $ such $d(f_{\#}([\alpha ]_{\delta }),[\beta
]_{\varepsilon })<\sigma $ in $Y_{\varepsilon }$.
\end{enumerate}
\end{proposition}

\begin{proof}
That the induced map is defined follows from $\delta +\sigma <\varepsilon $.
According to Lemma \ref{lestl}, up to $\delta $-homotopy and without
increasing the length of $\alpha $, we may assume that $\alpha
=\{x_{0},\dots,x_{n}\}$ where $n=\left\lfloor \frac{2L(\alpha )}{\delta }%
+1\right\rfloor $. We have 
\begin{equation*}
L(f(\alpha ))\leq L(\alpha )+n\sigma =L(\alpha )+\left( \frac{2L(\alpha )}{%
\delta }+1\right) \sigma.
\end{equation*}
By definition, $f(\alpha )\in f_{\#}(\left[ \alpha \right] _{\delta })$,
therefore $\left\vert f_{\#}([\alpha ]_{\delta })\right\vert \leq
L(f(\alpha))$. The proof of part 1 follows by taking the infimum of the
right side and using $\delta >\frac{\varepsilon }{2}$.

For the second part, fix $\tau >0$ and let $\alpha
^{\prime}:=\{x_{0},\dots,x_{m}\}$ be an $\omega $-chain such that $[\alpha
^{\prime}]_{\omega }=[\alpha ]_{\omega }$ and $L(\alpha ^{\prime })\leq
\left\vert\lbrack \alpha ]_{\omega }\right\vert +\tau $. By Lemma \ref{lestl}
we may suppose that $m=\left\lfloor \frac{2L(\alpha ^{\prime })}{\omega }%
+1\right\rfloor $. Then $\beta := f(\alpha
^{\prime})=\{f(x_{0}),\dots,f(x_{m})\}$ is an $\left( \omega +\sigma \right) 
$-chain of length at most $K=L(\alpha ^{\prime })+m\sigma $. Let $%
\eta=\left\langle \beta =\eta _{0},\dots,\eta _{n}=\beta ^{\prime
}\right\rangle $ be an $\left( \omega +\sigma \right) $-homotopy, where $%
L(\beta ^{\prime})\leq L(\beta )$; again by Lemma \ref{lestl} we may assume
that $\nu (\beta^{\prime })=m^{\prime }:=\left\lfloor \frac{2K}{\omega
+\sigma }+1\right\rfloor $. Denote by $y_{ij}$ the $i^{th}$ point of $\eta
_{j}$; note that the points $y_{ij}$ are not necessarily distinct!
Nonetheless, we may iteratively choose for each $y_{ij}$ a point $x_{ij}\in X
$ such that $d(f(x_{ij}),y_{ij})<\sigma $ and the following are true: $%
x_{i0}=x_{i}$ (which is possible since $\eta _{0}=\beta =f(\alpha )$), and
if $y_{ij}=y_{ab}$ then $x_{ij}=x_{ab}$. For any $x_{ij},x_{ab}$, 
\begin{equation*}
d(x_{ij},x_{ab})\leq d(f(x_{ij}),f(x_{ab}))+\sigma \leq
d(y_{ij},y_{ab})+3\sigma
\end{equation*}
and therefore if we let $\eta _{j}^{\prime }$ denote the chain in $X$ having 
$x_{ij}$ as its $i^{th}$ point, $\eta ^{\prime }:=\left\langle
\eta_{0}^{\prime },\dots,\eta _{n}^{\prime }\right\rangle $ is an $\left(
\omega +4\sigma \right) $-homotopy between $\alpha ^{\prime }$ and a chain $%
\alpha^{\prime \prime }$ of length at most $L(\beta ^{\prime })+m^{\prime
}\sigma $. Since $\omega +4\sigma <\delta $, $\eta ^{\prime }$ is in fact a $%
\delta $-homotopy. That is, 
\begin{equation*}
\left\vert \lbrack \alpha ]_{\delta }\right\vert \leq L(\beta
^{\prime})+\sigma \left( \frac{2K}{\omega +\sigma }+1\right) <L(\beta
^{\prime})+\sigma \left( \frac{4K}{\varepsilon }+1\right)
\end{equation*}
(in the second inequality we used $\omega +\sigma >\omega >\frac{\varepsilon 
}{2}$). Next, 
\begin{equation*}
K\leq L(\alpha ^{\prime })+\sigma \left( \frac{4L(\alpha ^{\prime })}{%
\varepsilon }+1\right) =L(\alpha ^{\prime })\left( \frac{4\sigma }{%
\varepsilon }+1\right) +\sigma \text{.}
\end{equation*}
Putting these two together yields: 
\begin{align*}
\left\vert \lbrack \alpha ]_{\delta }\right\vert &\leq L(\beta
^{\prime})+\sigma \left( \frac{4}{\varepsilon }+\frac{16\sigma }{\varepsilon
^{2}}\right) L(\alpha ^{\prime })+\sigma \left( \frac{4\sigma }{\varepsilon }%
+1)\right) \\
&\leq L(\beta ^{\prime })+\sigma \left( \frac{4}{\varepsilon }+\frac{%
16\sigma }{\varepsilon ^{2}}\right) \left( \left\vert [\alpha
]_{\omega}\right\vert +\tau \right) +\sigma \left( \frac{4\sigma }{%
\varepsilon }+1)\right)
\end{align*}
Letting $\tau \rightarrow 0$ gives us 
\begin{equation*}
L(\beta ^{\prime })\geq \left\vert \lbrack \alpha ]_{\delta }\right\vert
-\sigma \left( \frac{4}{\varepsilon }+\frac{16\sigma }{\varepsilon ^{2}}%
\right) \left( \left\vert [\alpha ]_{\omega }\right\vert \right) -\sigma
\left( \frac{4\sigma }{\varepsilon }+1)\right).
\end{equation*}
Taking the infimum over all $\beta ^{\prime }$ in $[f(\alpha
)]_{\omega+\sigma }$ yields the second inequality.

For the third part, let $\beta =\{\ast =y_{0},\dots,y_{n}\}$. As usual we
take $n:=\left\lfloor \frac{2L(\beta )}{\omega }+1\right\rfloor $. Since $f$
is a $\sigma $-isometry, we may find points $z_{i}=f(x_{i})$ such that $%
d(z_{i},y_{i})\leq \sigma $ for $1\leq i\leq n$. We have 
\begin{equation*}
d(x_{i},x_{i+1})\leq d(z_{i},z_{i+1})+\sigma \leq d(y_{i},y_{i+1})+\sigma
+2\sigma <\omega
\end{equation*}
and therefore the chain $\alpha :=\{\ast =x_{0},\dots,x_{n}\}$ is an $\omega 
$-chain such that $f(\alpha )=\beta ^{\prime }:=\{\ast =z_{0},\dots,z_{n}\}$%
. Moreover,

\begin{equation*}
L(\alpha )\leq L(\beta )+3\sigma \left( \frac{2L(\beta )}{\omega }+1\right)
\leq L(\beta )+3\sigma \left( \frac{4L(\beta )}{\varepsilon }+1\right).
\end{equation*}
Next, let $\beta ^{\prime \prime }:=\{\ast =z_{0},....,z_{n-1},z_{n},y_{n}\}$
and $\beta ^{\prime \prime \prime }:=\{\ast =y_{0},\dots ,y_{n},y_{n}\}$;
note that $[\beta ^{\prime \prime \prime }]_{\varepsilon
}=[\beta]_{\varepsilon }$. Since $\beta ^{\prime \prime \prime }$ is an $%
\left(\omega -3\sigma \right) $-chain, $E(\beta ^{\prime \prime
\prime})>\varepsilon -(\omega -3\sigma )>2\sigma $, and since $\Delta
(\beta^{\prime \prime \prime },\beta ^{\prime \prime })<\sigma $, we may
apply Proposition \ref{close} to conclude that $[\beta ^{\prime \prime
}]_{\varepsilon }=[\beta ^{\prime \prime \prime }]_{\varepsilon
}=[\beta]_{\varepsilon }$. Finally:

\begin{equation*}
d([\beta ^{\prime }]_{\varepsilon },[\beta ]_{\varepsilon
})=d([\beta^{\prime }]_{\varepsilon },[\beta ^{\prime \prime
}]_{\varepsilon})=\left\vert [\overline{\beta ^{\prime }}\ast \beta ^{\prime
\prime}]_{\varepsilon }\right\vert =\left\vert
[z_{n},y_{n}]_{\varepsilon}\right\vert =d(z_{n},y_{n})<\sigma
\end{equation*}
\end{proof}

\begin{theorem}
\label{firstdeg}For every $\varepsilon ,\sigma >0$ there is a first degree
polynomial $p(\sigma ,\varepsilon )$ with non-negative coefficients such
that $p \rightarrow 0$ as $\sigma \rightarrow 0$ (with $\varepsilon $ fixed)
and the following property holds. Let $X,Y$ be geodesic spaces such that $%
\phi _{\varepsilon \omega _{0}}:Y_{\omega _{0}}\rightarrow Y_{\varepsilon }$
is an injection and $\omega _{0}<\delta <\varepsilon $. If $f:X\rightarrow Y$
is a basepoint preserving $\sigma $-isometry for some positive constant $%
\sigma <\min \left\{ \varepsilon -\delta ,\frac{\delta -\omega _{0}}{4}%
\right\} $ then $f_{\#}:X_{\delta }\rightarrow Y_{\varepsilon }$ is a $%
p(\sigma ,\varepsilon )$-isometry.
\end{theorem}

\begin{proof}
Since $4\sigma <\delta -\omega _{0}$, we can choose $\omega $ such that $%
\omega _{0}<\omega <\delta -4\sigma $. Now we have the remaining two
conditions, $\omega <\delta $ and $\sigma <\frac{\delta -\omega }{4}$, that
are needed to apply Proposition \ref{serious}. Note that since $Y$ is
geodesic and hence all maps $\phi _{ab}$ are surjective, $\phi
_{\varepsilon\omega _{0}}$ is an isometry. But then $\phi _{\varepsilon
\omega }$ is also an isometry. Since $X$ is geodesic, we may always refine a 
$\delta $-chain to an $\omega $-chain of the same length, which means that
on the right side of the inequality in the second part of Proposition \ref%
{serious}, we may replace $\left\vert [\alpha ]_{\omega }\right\vert $ by $%
\left\vert [\alpha]_{\delta }\right\vert $. Also since $\phi _{\varepsilon
\omega }$ is an isometry, $\left\vert [f(\alpha )]_{\omega +\sigma
}\right\vert =\left\vert [f(\alpha )]_{\varepsilon }\right\vert =\left\vert
f_{\#}([\alpha ]_{\delta})\right\vert $. That is, we have 
\begin{equation*}
\left\vert \lbrack \alpha ]_{\delta }\right\vert -\left\vert
f_{\#}([\alpha]_{\delta })\right\vert \leq \sigma \left( \frac{4}{%
\varepsilon }+\frac{16\sigma }{\varepsilon ^{2}}\right) \left( \left\vert
[\alpha ]_{\delta}\right\vert \right) +\sigma \left( \frac{4\sigma }{%
\varepsilon }+1\right).
\end{equation*}
The first part of Proposition \ref{serious} gives us: 
\begin{equation*}
\left\vert f_{\#}([\alpha ]_{\delta })\right\vert -\left\vert
[\alpha]_{\delta }\right\vert \leq \frac{4\sigma }{\varepsilon }\left\vert
[\alpha]_{\delta }\right\vert +\sigma.
\end{equation*}
Let $m(\sigma ,\varepsilon ):=\sigma \left( \frac{4}{\varepsilon }+\frac{%
16\sigma }{\varepsilon ^{2}}\right) >\frac{4\sigma }{\varepsilon }$, $%
b(\sigma ,\varepsilon ):=\sigma \left( \frac{4\sigma }{\varepsilon }%
+1\right) >\sigma $ and $p(\sigma ,\varepsilon )(t)=m(\sigma
,\varepsilon)t+b(\sigma ,\varepsilon )$. The coefficients of $p$ then have
the desired property. Since 
\begin{equation*}
d(f_{\#}([\alpha ]_{\delta }),f_{\#}([\beta ]_{\delta
}))=d([f(\alpha)]_{\varepsilon },[f(\beta )]_{\varepsilon })=\left\vert 
\left[ \overline{f(\alpha )}\ast f(\beta )\right] _{\varepsilon }\right\vert
=\left\vert \left[ f(\overline{\alpha }\ast \beta )\right] _{\varepsilon
}\right\vert
\end{equation*}
and $d([\alpha ]_{\delta },[\beta ]_{\delta })=\left\vert [\overline{\alpha }%
\ast \beta ]_{\delta }\right\vert $ we see that $f_{\#}$ has distortion at
most $p(\sigma ,\varepsilon )$.

Finally, let $[\beta ]_{\varepsilon }\in Y_{\varepsilon }$; since $Y$ is
geodesic (and all maps $\phi _{ab}$ are surjective) there is some $%
\left(\omega -3\sigma \right) $-chain $\beta ^{\prime }$ such that $%
[\beta]_{\varepsilon }=[\beta ^{\prime }]_{\varepsilon }$. The proof is now
finished by the third part of Proposition \ref{serious}.
\end{proof}

\vspace{0.1in} We need an equivariant version of pointed Gromov-Hausdorff
convergence. Origins of such an idea may be found in \cite{F}, \cite{FY},
and something like this was used, for example, in \cite{EW}. Let $%
f:X\rightarrow Y$ be a function between metric spaces and suppose there are
groups $H$ and $K$ of isometries on $X$ and $Y$, respectively, with a
homomorphism $\psi :H\rightarrow K$. As usual, we say $f$ is equivariant
(with respect to $\psi $) if for all $h\in H$ and $x\in X$, $%
f(h(x))=\psi(h)(f(x))$. Then there is a well-defined induced mapping $%
f_{\pi}:X/H\rightarrow Y/K$ defined by $f_{\pi }(Hx)=Kf(x)$, where $%
Hx:=\{h(x):h\in H\}$ is the orbit of $x$. We will take the quotient
pseudo-metric on $X/H$ and $Y/K$: 
\begin{equation*}
d(Hx,Hy)=\inf \{d(k(x),h(y)):h,k\in H\}=\inf \{d(x,h(y)):h\in H\}\text{.}
\end{equation*}
When the orbits of the action are closed sets, $d$ is a \textit{bona fide}
metric.

In the next theorem note that some $\omega _{0}>\varepsilon $ to satisfy the
hypothesis always exists since the homotopy critical values are discrete.
The requirement that $\frac{\varepsilon }{2}<\omega _{0}$ isn't firm, but it
simplifies the calculation. All that really matters is that the statement is
true for every $\delta <\varepsilon $ that is sufficiently close to $%
\varepsilon $.

\begin{theorem}
\label{almstdone}Let $\{X_{i}\}$ be a collection of compact geodesic spaces
such that for all $i$ there is a basepoint preserving $\sigma _{i}$-isometry 
$f_{i}:X_{i}\rightarrow X$ for some sequence of constants $%
\sigma_{i}\rightarrow 0$. Let $\varepsilon >0$, suppose that $\phi
_{\varepsilon\omega _{0}}:X_{\omega _{0}}\rightarrow X_{\varepsilon }$ is
injective and let $\frac{\varepsilon }{2}<\omega _{0}<\delta <\varepsilon $.
Then for all large $i$, the following hold.

\begin{enumerate}
\item[\emph{1.}] $\left( f_{i}\right) _{\#}:\left( X_{i}\right)
_{\delta}\rightarrow X_{\varepsilon }$ is a $p(\sigma _{i},\varepsilon )$%
-isometry, and in particular, $(X_{i})_{\delta }$ is Gromov-Hausdorff
pointed convergent to $X_{\varepsilon }$.

\item[\emph{2.}] The restriction of $\left( f_{i}\right) _{\#}$ to $%
\pi_{\delta }(X_{i})$ - denoted hereafter by $(f_{i})_{\theta}$ - is an
isomorphism onto $\pi _{\varepsilon }(X)$.

\item[\emph{3.}] $\left( f_{i}\right) _{\#}$ is equivariant with respect to $%
(f_{i})_{\theta }$.
\end{enumerate}
\end{theorem}

\begin{proof}
The first part is an immediate consequence of Theorem \ref{firstdeg}. One
need only observe that $\sigma _{i}<\min \left\{ \varepsilon -\delta ,\frac{%
\delta -\omega _{0}}{4}\right\} $ for all large $i$. For the next two parts,
note that since $f_{i}$ is basepoint preserving, $(f_{i})_{\theta }$ does
map into $\pi _{\varepsilon }(X)$. That $(f_{i})_{\theta }$ is a
homomorphism, and that $\left( f_{i}\right) _{\#}$ is equivariant, both
follow from Equation (\ref{concat}), and it remains to be shown that $%
(f_{i})_{\theta }$ is an isomorphism for large $i$ . To do this we will add
a few more conditions that are satisfied for all large $i$. We first require 
$\sigma _{i}<\frac{\varepsilon }{6}$. Next note that if we let $%
p(\sigma_{i},\varepsilon )(t):=m_{i}t+b_{i}$, then $m_{i},b_{i}\rightarrow 0$%
. If $d\bigl(\left( f_{i}\right) _{\#}([\alpha ]_{\delta}),\left(
f_{i}\right)_{\#}([\beta ]_{\delta })\bigr)=D$ then $d([\alpha ]_{\delta
},[\beta]_{\delta })\leq \frac{D+b_{i}}{1-m_{i}}$. In particular, we may
conclude the following for all large $i$: 
\begin{equation}  \label{ifthen}
\text{If }d(\left( f_{i}\right) _{\#}([\alpha ]_{\delta
}),\left(f_{i}\right) _{\#}([\beta ]_{\delta }))<\frac{\varepsilon }{3}, 
\text{ then } d([\alpha ]_{\delta },[\beta ]_{\delta })<\frac{\varepsilon }{2%
}\text{.}
\end{equation}
For large enough $i$ the following also hold: 
\begin{equation}  \label{ifthen3}
\text{If }[\beta ]_{\varepsilon }\in X_{\varepsilon},\text{ then there is
some }[\alpha ]_{\delta }\text{ such that }d(\left(
f_{i}\right)_{\#}([\alpha ]_{\delta }),[\beta ]_{\varepsilon })<\frac{%
\varepsilon }{6}.
\end{equation}
\begin{equation}  \label{ifthen2}
\text{If }d([\alpha ]_{\delta },[\beta ]_{\delta })<\frac{\varepsilon }{3}, 
\text{ then }d(\left( f_{i}\right) _{\#}([\alpha ]_{\delta
}),\left(f_{i}\right) _{\#}([\beta ]_{\delta }))<\frac{\varepsilon }{2}.
\end{equation}

Suppose that $[\lambda ]_{\delta }\in \ker (f_{i})_{\theta }$. Then $%
d((f_{i})_{\theta }\left( [\lambda ]_{\delta }\right) ,[\ast
]_{\varepsilon})=0<\frac{\varepsilon }{3}$ and by (\ref{ifthen}), $%
d([\lambda ]_{\delta},[\ast ]_{\delta })<\frac{\varepsilon }{2}<\delta $.
But $\phi _{\delta }$ is injective on $B(\ast ,\delta )$ and since $\lambda $
is a loop, $[\lambda]_{\delta }=[\ast ]_{\delta }$.

Let $[\lambda ]_{\varepsilon }\in \pi _{\varepsilon }(X)$. By (\ref{ifthen3}%
) there is $[\alpha ]_{\delta }$ such that $d\bigl(\left(f_{i}\right)
_{\#}([\alpha ]_{\delta }),[\lambda ]_{\varepsilon }\bigr)<\frac{\varepsilon 
}{6}$. Letting $\alpha :=\{\ast =x_{0},\dots,x_{n}\}$ we have 
\begin{equation*}
d(x_{n},\ast )\leq d(f(x_{n}),\ast )+\sigma _{i}\leq d(\left(
f_{i}\right)_{\#}([\alpha ]_{\delta }),[\lambda ]_{\varepsilon })+\sigma
_{i}<\frac{\varepsilon }{3}<\delta \text{.}
\end{equation*}
Then $\lambda ^{\prime }:=\{\ast =x_{0},\dots,x_{n},\ast \}$ is a $\delta $%
-loop with $d([\lambda ^{\prime }]_{\delta },[\alpha ]_{\delta })\leq
d(x_{n},\ast )<\frac{\varepsilon }{3}$. Therefore by (\ref{ifthen2}) $d\bigl(%
\left( f_{i}\right) _{\#}([\alpha ]_{\delta }),\left(
f_{i}\right)_{\#}([\lambda ^{\prime }]_{\delta })\bigr)<\frac{\varepsilon }{2%
}$. The triangle inequality now shows $d\bigl(\lbrack \lambda
]_{\varepsilon},\left( f_{i}\right) _{\#}([\lambda ^{\prime }]_{\delta })%
\bigr)<\frac{5\varepsilon }{6}<\varepsilon $ . But once again, the
injectivity of $\phi_{\varepsilon }$ on open $\varepsilon $-balls implies
that $[\lambda]_{\varepsilon }=\left( f_{i}\right) _{\#}([\lambda ^{\prime
}]_{\delta})=(f_{i})_{\theta }([\lambda ^{\prime }]_{\delta })$, finishing
the proof of surjectivity.
\end{proof}

\begin{proposition}
\label{equiclose}Suppose that $X$ and $Y$ are metric spaces with groups $H,K$
acting on $X,Y$, respectively, by isometries with closed orbits and $%
\phi:H\rightarrow K$ is an epimorphism. If $f:X\rightarrow Y$ is a $\sigma $%
-isometry for some first degree polynomial $\sigma $ and equivariant with
respect to $\phi $, then $f_{\pi }:X/H\rightarrow Y/K$ is a $\sigma $%
-isometry.
\end{proposition}


\begin{proof}
For any $x,y\in X$, we have $D := d(f_{\pi }(Hx),f_{\pi}(Hy))=d(Kf(x),Kf(y)) 
$, from which we obtain 
\begin{align*}
D &=\inf \{d(f(x),h(f(y))):h\in K\} \\
&=\inf \{d(f(x),\phi (g)(f(y))):g\in H\} \\
&=\inf \{d(f(x),f(g(y))):g\in H\}
\end{align*}
(The first equality follows because $\phi $ is surjective.) Now for any $%
g\in H$, 
\begin{equation*}
d(x,g(y))-\sigma (d(x,g(y))\leq d(f(x),f(g(y)))\leq d(x,g(y))+\sigma
(d(x,g(y)).
\end{equation*}
Letting $D^{\prime }:=d(Hx,Hy)=\inf \left\{ d(x,g(y))\right\} $, the infimum
of the right side is $D^{\prime }+\sigma (D^{\prime })$. For the left side,
for arbitrary $\varepsilon >0$ we may suppose that 
\begin{equation*}
D^{\prime }\leq d(x,g(y))<D^{\prime }+\varepsilon
\end{equation*}
which gives us 
\begin{equation*}
d(x,g(y))-\sigma (d(x,g(y))>D^{\prime }-\sigma (D^{\prime }+\varepsilon)
\end{equation*}
and therefore the infimum of the left side is $D^{\prime }-\sigma(D)^{\prime
}$. That is, the distortion of $f_{\pi }$ is at most $\sigma $.

Finally, for any $Ky\in Y/K$, there is some $f(x)=z\in Y$ such that $%
d(z,y)\leq \sigma (0)$. But then by definition $Kz=f(Kx)$ and $d(Kz,Ky)\leq
d(z,y)\leq \sigma (0)$.\bigskip
\end{proof}

Next, we prove a special instance of Theorem \ref{closed} in the case where $%
\varepsilon $ is not a critical value of the limit space. This result can be
extracted from a combination of results established by Sormani and Wei in 
\cite{SW1}. We include it here as a proposition in part because it has not
yet been stated elsewhere in this specific form, and also because we provide
an alternative proof using our discrete methods.

\begin{proposition}
\label{GHC}Suppose that $X_{i}\rightarrow X$, where each $X_{i}$ is compact
geodesic and let $\varepsilon >0$. Then for any $\delta <\varepsilon $
sufficiently close to $\varepsilon $, $(X_{i})_{\delta }\rightarrow
(X)_{\varepsilon }$ and $\pi _{\delta }(X_{i})$ is isomorphic to $%
\pi_{\varepsilon }(X)$ for all large $i$. In particular, if $\varepsilon $
is not a homotopy critical value of $X$ then $(X_{i})_{\varepsilon
}\rightarrow X_{\varepsilon }$ and $\pi _{\varepsilon }(X_{i})$ is
eventually isomorphic to $\pi _{\varepsilon }(X)$.
\end{proposition}

\begin{proof}
Let $f_{i}:X_{i}\rightarrow X$ be basepoint-preserving $\sigma _{i}$%
-isometries with constants $\sigma _{i}\rightarrow 0$. Since the homotopy
critical values of $X$ are discrete, we may choose $\omega _{0}$, and hence $%
\delta $ with $\frac{\varepsilon }{2}<\delta <\varepsilon $, so that the
assumptions of Theorem \ref{almstdone} are satisfied. Eliminating finitely
many terms if needed, we obtain the following properties for all $i$ that we
will use now and below: (1) $\left( f_{i}\right) _{\#}:\left(
X_{i}\right)_{\delta }\rightarrow X_{\varepsilon }$ is a $p(\sigma
_{i},\varepsilon )$-isometry. (2) The restriction $(f_{i})_{\theta }$ of $%
\left( f_{i}\right)_{\#}$ to $\pi _{\delta }(X_{i})$ is an isomorphism onto $%
\pi _{\varepsilon}(X)$. (3) $\left( f_{i}\right) _{\#}$ is equivariant with
respect to $(f_{i})_{\theta }$. The first statement of Theorem 1 is an
immediate consequence of (1) and (2). If $\varepsilon $ is not a homotopy
critical value of $X$ then there are some $\varepsilon ^{\prime
}>\varepsilon >\omega_{0}$ such that $\phi _{\varepsilon ^{\prime }\omega
_{0}}$ is an isometry. We may now apply the first part of the theorem using $%
\varepsilon ^{\prime }$ to see that $(X_{i})_{\varepsilon }\rightarrow
(X)_{\varepsilon ^{\prime }}$ and $\pi _{\varepsilon }(X_{i})$ is eventually
isomorphic to $\pi_{\varepsilon ^{\prime }}(X)$. But since $\phi
_{\varepsilon ^{\prime}\varepsilon }:X_{\varepsilon }\rightarrow
X_{\varepsilon ^{\prime }}$ is an isometry, $\theta _{\varepsilon ^{\prime
}\varepsilon }:\pi _{\varepsilon}(X)\rightarrow \pi _{\varepsilon ^{\prime
}}(X)$ is an isomorphism. This completes the proof.
\end{proof}

\begin{proof}[Proof of Theorem \protect\ref{closed}]
For the proof of Theorem \ref{closed} we will continue with the same
notation and $\delta $ as chosen in the proof of Proposition \ref{GHC}. By
eliminating terms if needed we may assume that $\varepsilon _{i}>\delta $
for all $i$. By Proposition \ref{T} the covering space $X_{\varepsilon
_{i}}^{\mathcal{T}_{i}}$ is isometrically equivalent to $X_{\delta }^{%
\mathcal{S}_{i}}$, and $\pi_{\varepsilon _{i}}^{\mathcal{T}_{i}}(X)$ is
isomorphic to $\pi _{\delta }^{\mathcal{S}_{i}}(X)$, where $\mathcal{S}_{i}$
is obtained by adding to $\mathcal{T}_{i}$ one representative for each
essential $\tau $-circle with $\varepsilon _{i}>\tau \geq \varepsilon $.
Therefore we need only show that there is some collection $\mathcal{T}$ as
in the statement of the theorem, and a subsequence such that $%
(X_{i_{k}})_{\delta _{i_{k}}}^{\mathcal{S}_{i_{k}}}\rightarrow
X_{\varepsilon }^{\mathcal{T}}$ and $\pi _{\delta_{i_{k}}}^{\mathcal{S}%
_{i_{k}}}(X_{i_{k}})$ is eventually isomorphic to $\pi_{\varepsilon }^{%
\mathcal{T}}(X)$. Let $g_{i}:(X_{i})_{\delta }\rightarrow X_{\delta }$
denote $\phi _{\varepsilon \delta }^{-1}\circ (f_{i})_{\#}$, which is a $%
p(\sigma _{i},\varepsilon )$-isometry, and let $h_{i}$ be the restriction of 
$g_{i}$ to $\pi _{\delta }(X_{i})$. By Conditions (2) and (3) above, $h_{i}$
is an isomorphism onto $\pi _{\delta }(X)$ for all $i$, and the maps $g_{i}$
are invariant with respect to the isomorphisms $h_{i}$. According to Theorem
11 of \cite{PW}, the number of homotopy critical values $\geq \varepsilon $
(counted with multiplicity) in the Gromov-Hausdorff precompact collection $%
\{X_{i}\}$ has a uniform upper bound. Therefore by removing equivalent
essential triads and taking a subsequence if necessary, we may assume that
for some $n$, $\mathcal{S}_{i}=\{T_{i1},\dots ,T_{in}\}$ ($n$ could be $0$,
in which case the following statements about $\mathcal{T}_{i}$ are true for
the empty set). Suppose that $T_{ij}=\{x_{ij}^{0},x_{ij}^{1},x_{ij}^{2}\}$
is a $\delta _{ij}$-triad and $T_{ij}^{\prime }$ is an $\frac{\varepsilon }{3%
}$-refinement of $\alpha_{T_{ij}}$. Since the diameters of the spaces $X_{i}$
have a uniform upper bound, the number of points needed to refine each $%
\alpha _{T_{ij}}$ has a uniform upper bound; by adding points if necessary
we may assume that for some fixed $w$, $T_{ij}^{\prime
}=\{z_{ij}^{0}=x_{ij}^{0},\dots,z_{ij}^{w}=x_{ij}^{0}\}$ for all $i,j$. The
uniform upper bound on diameters also implies that for some fixed $m$ we may
find $\frac{\varepsilon }{3}$-chains $\alpha _{ij}:=\{\ast =y_{ij}^{0},\dots
,y_{ij}^{m}=x_{ij}^{0}\}$ for all $i,j$ (i.e. subdivide geodesics). Finally,
let $\lambda _{ij}:=\alpha _{ij}\ast T_{ij}^{\prime }\ast \overline{%
\alpha_{ij}}$.

By choosing a subsequence yet again we may suppose that for all $j,k$, $%
f_{i}(z_{ij}^{k})\rightarrow z_{j}^{k}$, $f_{i}(x_{ij}^{k})\rightarrow
x_{j}^{k}$ and $f_{i}(y_{ij}^{k})\rightarrow y_{j}^{k}$. Let $%
\alpha_{j}:=\{y_{j}^{0},\dots ,y_{j}^{m}\}$, $T_{j}:=%
\{x_{j}^{0},x_{j}^{1},x_{j}^{2}\}$, $T_{j}^{\prime
}:=\{z_{j}^{0},\dots,z_{j}^{w}\}$, and $\lambda _{j}:=\alpha _{j}\ast
T_{j}^{\prime }\ast \overline{\alpha _{j}}$. Since all the limiting chains
have the property that each point is of distance at most $\frac{\varepsilon 
}{3}$ from its successor, Proposition \ref{close} implies that there is some 
$N$ such that if $i\geq N$, then 
\begin{equation}
\lbrack f_{i}(\lambda _{ij})]_{\rho }=[\lambda _{j}]_{\rho }\text{ for all }%
j\ \text{and any }\rho \geq \frac{\varepsilon }{2}>\frac{\varepsilon }{3}%
\text{.\ }  \label{rho}
\end{equation}
We assume $i\geq N$ in what follows. One immediate consequence of (\ref{rho}%
) is that 
\begin{equation}
(f_{i})_{\theta }\bigl(\lbrack \lambda _{ij}]_{\delta }\bigr)=(f_{i})_{\#}%
\bigl(\lbrack \lambda _{ij}]_{\delta }\bigr)=[f(\lambda
_{ij})]_{\varepsilon}=[\lambda _{j}]_{\varepsilon }  \label{fker}
\end{equation}
where $\left( f_{i}\right) _{\theta }$ is the restriction of $(f_{i})_{\#}$
to $\pi _{\delta }(X_{i})$.

We can now argue that $T_{j}$ is an essential $\delta _{j}$-triad where $%
\delta _{j}:=\lim \delta _{ij}\geq \varepsilon $. In fact, by continuity of
the distance function, $\delta _{j}=\lim \delta _{ij}\geq \varepsilon $
exists and $T_{j}$ is a $\delta _{j}$-triad. If $T_{j}^{\prime }$ were $%
\varepsilon $-null then $f_{i}(\lambda _{ij})$ would be also $\varepsilon $%
-null, which by (\ref{fker}) means $[\lambda _{ij}]_{\delta }\in \ker $ $%
\left( f_{i}\right) _{\theta }$. Since $\left( f_{i}\right) _{\theta }$ is
an isomorphism, $\lambda _{ij}$, and hence $T_{ij}^{\prime }$, is $%
\varepsilon $-null. This contradicts that $T_{ij}$ is $\delta _{ij}$%
-essential with $\delta _{ij}>\varepsilon $. So $T_{j}^{\prime }$ is not $%
\varepsilon $-null, and $\delta _{j}\geq \varepsilon $ by Corollary \ref%
{three}. If $T_{j}$ were not essential then $T_{j}^{\prime }$ would be $%
\delta _{j}$-null. Then for some $\delta ^{\prime }<\delta _{j}$ and close
enough to $\delta _{j}$, $T^{\prime }$ would also be $\delta ^{\prime }$%
-null. But then for large enough $i$, $\delta _{ij}>\delta ^{\prime }$ and (%
\ref{rho}) implies $T_{ij}^{\prime }$ is $\delta ^{\prime }$-null, a
contradiction to the fact that $T_{ij}^{\prime }$ is not $\delta _{ij}$-null.

The next consequence of (\ref{fker}), and the characterization of $%
K_{\delta}(\mathcal{T})$ in Remark \ref{comments}, is that $h_{i}(K_{\delta
}(\mathcal{S}_{i}))=K_{\delta }(\mathcal{T)}$, where $\mathcal{T}%
:=\{T_{1},\dots ,T_{n}\}$. At this point we may assume the following, having
chosen subsequences several times (but avoiding double subscripts for
simplicity): the functions $g_{i}:(X_{i})_{\delta }\rightarrow X_{\delta }$
are $p(\sigma _{i},\varepsilon )$-isometries and the restrictions $k_{i}$ of 
$h_{i}$ to $K_{\delta }(\mathcal{S}_{i})$ are isometries onto $K_{\delta }(%
\mathcal{T})$ that are equivariant with respect to $g_{i}$. By Propositions %
\ref{ghpoly} and \ref{equiclose}, $(X_{i})_{\delta }^{\mathcal{S}%
_{i}}=\left( X_{i}\right) _{\delta }/K_{\delta }(\mathcal{S}_{i})\rightarrow
X_{\delta }/K_{\delta }(\mathcal{T)=}X_{\delta }^{\mathcal{T}}$. By the
choice of $\delta $, $\phi _{\varepsilon \delta }$ is an isometry, so $%
X_{\delta }/K_{\delta }(\mathcal{T)}$ is isometric to $X_{\varepsilon}/K_{%
\varepsilon }(\mathcal{T)=}X_{\varepsilon }^{\mathcal{T}}$ by Proposition %
\ref{T}.

Finally, recall that $h_{i}:\pi _{\delta }(X_{i})\rightarrow \pi _{\delta}(X)
$ is an isomorphism that takes $K_{\delta }(\mathcal{S}_{i})$ to $K_{\delta
}(\mathcal{T)}$. Combining this with the first part of Theorem \ref{circles}
gives us that $\pi _{\delta }^{\mathcal{S}_{i}}(X_{i})=\pi_{\delta
}(X_{i})/K_{\delta }(\mathcal{S}_{i})$ is isomorphic to $\pi_{\delta
}(X)/K_{\delta }(\mathcal{T})$. But $\theta _{\varepsilon \delta }$ is an
isomorphism from $\pi _{\delta }(X)$ to $\pi _{\varepsilon }(X)$ taking $%
K_{\delta }(\mathcal{T})$ to $K_{\varepsilon }(\mathcal{T})$, completing the
proof of the theorem.
\end{proof}

\section{Some Open Questions and Problems}

There are some questions that naturally arise from these results and might
make interesting motivations for future work. For example, is it possible to
characterize circle covers among all covers (this extends a question from 
\cite{PW} about characterizing $\varepsilon $-covers)? We know that not all
covers of a compact geodesic space are circle covers. For example, the
standard geodesic circle has only two circle covers: the trivial cover and
the universal cover; other non-equivalent covers like the double cover
cannot be circle covers. But at this point we are only able to identify when
a cover is a circle cover in the following ways: (1) by exclusion when we
know all the circle covers in a particular example, (2) if the cover is
explicitly defined as a circle cover, or (3) if it is known to be so by
Theorem \ref{closed}. In this connection we note that the natural analog of
Theorem \ref{closed} for covering maps in general is not true. In fact, for
a circle cover $\pi $ of a compact geodesic space $X$, a lower bound $%
\varepsilon >0$ on the size of the circles is equivalent to being covered by
the $\varepsilon $-cover of $X$. Now let $\psi_{k}:C_{k}\rightarrow C_{1}$
be the $k$-fold cover of the geodesic circle $C_{1}$ by the geodesic circle
of length $k$, which as we have mentioned is not a circle cover for $k>1$.
Each of these covers is covered by the universal covering space of $C_{1}$,
which is the $\frac{1}{3}$-cover of $C_{1}$. It is also true that $%
C_{k}\rightarrow \mathbb{R}$ in the pointed Gromov-Hausdorff sense (in fact
it is not hard to show in general precompactness of covering spaces covered
by an $\varepsilon $-cover). However, the deck groups of these covering maps
are $\mathbb{Z}_{k}$, which of course are all distinct.

Another question of interest is related to the fairly old question
concerning the degree to which various spectra (Laplacian, length, covering)
determine geometric properties in a compact geodesic space, including
whether they must be isometric. Note that, up to a multiplied constant, the
covering and homotopy critical spectra are contained in the length spectrum.
While this was already observed by Sormani-Wei in \cite{SW2}, this is an
immediate consequence of the previously mentioned fact that $X$ contains an
essential $\varepsilon $-circle if and only if $\varepsilon $ is a homotopy
critical value. The relationship between the length and the Laplace spectra
was first considered in \cite{BB2}, \cite{dV}, \cite{DG}. Already de Smit,
Gornet, and Sutton have shown that the covering spectrum is not a spectral
invariant (\cite{dGS}, \cite{DGS2}) by extending Sunada's method \cite{Su}
to determine when two manifolds have the same Laplace spectrum. However,
essential circles allow one to enhance the notion of covering/homotopy
critical spectrum in the following way. Given a compact geodesic space $X$,
each circle covering of $X$ corresponds to a subgroup of $\pi _{1}(X)$,
which we will call a \textit{circle group}. Specifically, a circle group is
the kernel of the natural map $\Lambda :\pi _{1}(X)\rightarrow \pi
_{\varepsilon }(X)$ mentioned in Section 2, composed with the quotient map
from $\pi _{\varepsilon }(X)$ to $\pi _{\varepsilon }(X)/K_{\varepsilon }(%
\mathcal{T)}$ described above. The collection of all circle groups,
partially ordered by inclusion, provides an algebraic refinement of the
homotopy critical spectrum which in principle should say more about how
similar two spaces are. That is, what can be said about compact geodesic
spaces that not only share the same homotopy critical spectra, including
multiplicity, but also share the same partially ordered collection of circle
groups up to isomorphism?

\end{document}